\documentclass[12pt,a4paper]{amsart}
\usepackage{amssymb}
\usepackage{amsmath}
\usepackage{xcolor}
\usepackage{cite}
\usepackage{amscd}
\usepackage{pstricks}
\usepackage{pst-node}

\usepackage{geometry}
\newtheorem{thm}{Theorem}[section]
\newtheorem{cor}[thm]{Corollary}
\newtheorem{lem}[thm]{Lemma}
\newtheorem{prop}[thm]{Proposition}

\theoremstyle{definition}

\usepackage{graphics}
\usepackage{graphicx}

\newtheorem{rem}[thm]{Remark}

\numberwithin{equation}{section}

\begin{document}
\subjclass[2010]{}

\title[Complete integrability of subriemannian geodesic flow on $\mathbb{S}^7$]{Complete integrability of subriemannian \\ geodesic flows on $\mathbb{S}^7$}

\author{Wolfram Bauer, Abdellah Laaroussi, Daisuke Tarama}

\thanks{} 

\address{Wolfram Bauer\endgraf
Institut f\"{u}r Analysis, Leibniz Universit\"{a}t  \endgraf
Welfengarten 1, 30167 Hannover, Germany \endgraf
}
\email{bauer@math.uni-hannover.de}

\address{Abdellah Laaroussi \endgraf
Institut f\"{u}r Analysis, Leibniz Universit\"{a}t  \endgraf
Welfengarten 1, 30167 Hannover, Germany \endgraf
}
\email{laarousi.abdellah@gmail.com}

\address{Daisuke Tarama \footnote{Partially supported by JSPS KAKENHI Grant Numbers JP19K14540, JP22H01138, JP23H04481.}\endgraf
Department of Mathematical Sciences, Ritsumeikan University, \endgraf
1-1-1 Nojihigashi, Kusatsu, Shiga, 525-8577, Japan. 
}
\email{dtarama@fc.ritsumei.ac.jp}

 \date{\today}
\begin{abstract}
Four subriemannian (SR) structures over the Euclidean sphere $\mathbb{S}^7$ are considered in accordance to the previous literature. 
The defining bracket generating distribution is chosen as  the horizontal space in 
the Hopf fibration, the  quaternionic Hopf fibration or spanned by a suitable number of canonical vector fields.
In all cases the induced SR geodesic flow on $T^*\mathbb{S}^7$ is studied. 
Adapting a method by A. Thimm in \cite{T}, a maximal set of functionally
independent and Poisson commuting first integrals are constructed, including the corresponding SR Hamiltonian. 
As a result, the complete integrability  in the sense of Liouville is proved for the SR geodesic flow. 
It is observed that these first integrals arise as the symbols of commuting second order differential operators one of them being a (not necessarily intrinsic) sublaplacian. 
On the way one explicitly derives the Lie algebras of all SR isometry groups intersected with $O(8)$. 
\vspace{1mm}\\
{\bf keywords:} subriemannian manifold, Thimm's method, commuting differential operators
\vspace{1mm}
\end{abstract}
\maketitle
\tableofcontents
\thispagestyle{empty}
\section{Introduction}\label{S_Intro}
\label{section-1}
Subriemannian (SR) geometry provides a mathematical framework for the study of motions under non-holonomic constraints. 
A subriemannian manifold $M$ means a triple $(M, \mathcal{D}, g)$ consisting of a smooth manifold $M$, a bracket generating tangential distribution $\mathcal{D}$ over $M$ and a family 
of inner products $g=\left(g_p\right)_{p\in M}$ on $\mathcal{D}$ smoothly varying with the base point $p \in M$.  All subriemannian manifolds appearing in the present paper are equiregular of step two 
(see Section \ref{section-2} or \cite{ABB, Mo} for definitions and more details). 

Similar to the case of a Riemannian manifold (i.e. the case $\mathcal{D}=TM$) one assigns to $M$ a number of mathematical objects that are closely linked to its geometric structure: to each point of $M$ one can attach a local model ({\it tangent cone}) which is a stratified Lie group and generalizes the tangent space in Riemannian geometry. 
Naturally, $(M, \mathcal{D}, g)$ carries a metric structure inducing a Hausdorff dimension which often plays a role in the analysis of induced geometric operators on $M$, e.g. 
see \cite{BL,Bena,Verd}.

From an analytical point of view it has been observed (e.g. see \cite{ABB,AGBR, BR,Hoermander}) that an equiregular subriemannian manifold carries an intrinsic sublaplacian $\Delta_{\textup{sub}}$ which, based on the bracket generating property of $\mathcal{D}$, is a hypoelliptic second order differential operator. 
In what follows, the principle symbol $H_{\textup{sub}}$ of $-\frac{1}{2}\Delta_{\textup{sub}}$ is called the {\it subriemannian Hamiltonian}. 
In a standard way it induces the geodesic flow with respect to the SR structure on the cotangent bundle $T^*M$. A rich interplay between these objects has been discovered by many authors and led to interesting links between analysis and geometry 
\cite{BauWang-1,BauWang,BL,BLT,Bena,Verd}. 
The degeneracy of the SR Hamiltonian is one source of difficulties that arise in the analysis of the sublaplacian. In fact, it causes analytic or geometric 
effects that cannot be observed in Riemannian geometry.  
\vspace{1mm}\par 
Since the analysis on $M$ intimately depends on the particular nature of the SR structure one often poses further assumptions or studies explicit model cases. 
In the present paper  $M$ is chosen to be the Euclidean unit sphere $\mathbb{S}^7$ and in all cases $g$ is defined as the restriction to 
$\mathcal{D}$ of the standard Riemannian metric on $\mathbb{S}^7$. 
Hence by varying $\mathcal{D}$ one obtains a family of SR structures on $\mathbb{S}^7$. 
Two of these SR geometries are induced in a rather standard way (see \cite{Mo}) from principle bundle structures on $\mathbb{S}^7$. More precisely, 
$\mathcal{D}$ is defined as the horizontal space in the Hopf fibration and the quaternionic Hopf fibration, respectively. As is well known, the horizontal distributions $\mathcal{D}$ are 
bracket generating of step 2, see \cite{BFI-0,MaMo12,Mo}. The construction of these SR structures generalizes to Euclidean spheres of dimension $2n+1$ and $4n+3$ with $n \in \mathbb{N}$, respectively.  Among other results the heat kernel of the intrinsic sublaplacians $\Delta_{\textup{sub}}$ as well as its asymptotic properties are explicitly known in 
both cases and allow a spectral decomposition of $\Delta_{\textup{sub}}$, see \cite{BauWang, BauWang-1,PGreiner}. 
\vspace{1mm}\par 
As was observed in \cite{BFI} the sphere $\mathbb{S}^7$ carries two other SR geometries of rank $4$ and $5$, respectively, which were called {\it trivializable}. 
In fact, in these examples the distribution $\mathcal{D}$ is trivial as a vector bundle and spanned by a choice of canonical vector fields induced from a Clifford module 
action on $\mathbb{R}^8$, see \cite{Adams,James}. 
For the convenience of the reader, the construction is described in Section \ref{section-2} below. 
The recent paper \cite{BL}  compares the SR geometries on $\mathbb{S}^7$ induced by the horizontal space in the quaternionic Hopf fibration and 
the rank 4 trivializable SR structure with respect to their geometric and analytical properties. Both geometries are of rank 4 and step 2 but non-isometric in subriemannian sense and, in short, of quite different nature. As for the rank 5 trivializable SR structure the heat kernel and the fundamental solution of the {\it conformal sublaplacian} have been derived explicitly in \cite{BLT}. 
\vspace{1mm}\par
The current paper presents an analysis of the {\it subriemannian geodesic flow} simultaneously for all the above SR structures on $\mathbb{S}^7$. 
Although, in the case of a trivializable SR structure,  one can work with global affine coordinates of $\mathbb{R}^8$, 
it seems to be a non-trivial task to find all solutions to Hamilton's equation explicitly. 
To analyze the explicit solvability of the geodesic flows, one can use the notion of Liouville integrability. 
A Hamiltonian system on a $2n$-dimensional symplectic manifold is called \textit{completely integrable} if 
there exist $n$ functionally independent and Poisson commuting first integrals. 
Then, under suitable conditions, {\it Liouville-Arnol'd Theorem} states that the momentum mapping (cf. Section \ref{Section-4}) 
given as the collection of $n$ independent first integrals defines a 
torus bundle structure around the regular fibres. The restriction of the flow to each generic fibre turns out to be linear and hence can be solved by quadrature in principle. 
This serves as a geometric interpretation of solvability of a completely integrable system. Various methods of finding a maximal number of first integrals such as the use of Lax equations (see e.g. \cite{Fomenko,Perelomov,T}) are known in the literature. 

\vspace{1mm}\par 
Our analysis essentially 
extends preliminary results in \cite{RF} on $\mathbb{S}^7$ and it is along the same line of the paper \cite{RT_pseudo_H} which concerns the complete integrability of the geodesic flows on pseudo-$H$-type Lie groups. 
As a novelty, we adapt {\it Thimm's method} in \cite{T} for proving complete integrability for 
a specific collection of Hamiltonian systems to the setting of SR geometries on $\mathbb{S}^7$. 
In each of the above cases we construct an adapted increasing chain
\begin{equation}\label{Introduction_chain_of_subalgebras}
\mathfrak{a}_1 \hookrightarrow \mathfrak{a}_2 \hookrightarrow \cdots \hookrightarrow \mathfrak{a}_6 \hookrightarrow \mathfrak{so}(8)
\end{equation}
of non-degenerate Lie subalgebra of $\mathfrak{so}(8)$ which at some place includes the Lie algebra of the SR isometry group (intersected with $O(8)$). {\it Thimm's method} 
now allows to construct seven first integrals $f_1, \ldots, f_7$ in involution and including the SR Hamiltonian from the chain (\ref{Introduction_chain_of_subalgebras}). 
After the identification of the tangent and cotangent space via the standard Riemannian metric on $\mathbb{S}^7$ and applying the inclusion $T\mathbb{S}^7\subset \mathbb{R}^8 \times \mathbb{R}^8$ it is shown that $f_j$, with $j=1, \ldots, 7$ can be chosen as the restriction of explicitly given homogeneous polynomials on $\mathbb{R}^8 \times \mathbb{R}^8$.
In a second step one verifies the functional independence of the first integrals in all cases proving the complete integrability of the SR structures under 
consideration.

It seems to be a challenging problem to determine the SR isometry groups $\textup{Iso}_{\textup{sub}}(\mathbb{S}^7)$. In particular, these groups are of interest in the 
heat kernel analysis since their action leaves the kernel invariant and therefore allows to reduce the number of coordinates, see \cite{BauWang-1,BauWang, BLT}. 
For all the examples appearing in the present paper, the Lie algebra of the intersection $\textup{Iso}_{\textup{sub}}(\mathbb{S}^7)\cap O(8)$ is described explicitly. 
\vspace{1mm}\par 
Poisson commuting first integrals can also be obtained from the principal symbols of commuting (pseudo)-differential operators. 
In the setting of the present paper, we show a converse statement. 
In each case, a set of seven pairwise commuting second order differential operators is constructed. 
The set of operators acts on $\mathbb{S}^7$ and includes the sublaplacian of the associated SR geometry.  
The set of symbols of these operators coincides with the set of the previously constructed functionally independent first integrals $f_j$, $j=1, \dots, 7$. 
\vspace{1mm} \par 
The structure of the paper is as follows:\\
Section 2 recalls the construction of the SR structures on $\mathbb{S}^7$ which are dealt with in the present paper, i.e. the trivializable SR structures of rank $4, 5, 6$ and the Hopf and the quaternionic Hopf fibrations. 
In Section 3, examples of non-periodic normal geodesics are given in the case of the trivializable SR structures of rank $4$ and $5$. 
The complete integrability of the SR geodesic flow is discussed in Section 4 for each of the above SR structures. 
The key technique is Thimm's method which is based on an explicit choice of non-degenerate subalgebras of the Lie algebra $\mathfrak{so}(8)$. These algebras again play a key role in the descriptions of the infinitesimal SR isometries restricted to the orthogonal group $O(8)$ in Section 5. 
In the end of the paper, commuting differential operators are constructed in association to the sublaplacians for each of the above SR structures and the relation to the SR geodesic flows is explained. 
\section{Subriemannian structures on $\mathbb{S}^7$}
\label{section-2}
We recall the construction of the four subriemannian (SR) structures defined on the Euclidean sphere $\mathbb{S}^7$ which previously have been considered in  \cite{BauWang-1,BFI,PGreiner}. 
The defining distribution $\mathcal{D}$ is either induced by a set of canonical vector fields or given as the horizontal space in the Hopf or quaternionic Hopf fibration. 
In each case, $\mathcal{D}$ is  bracket generating, i.e. the Lie hull of vector fields taking values in $\mathcal{D}$ evaluated 
at any point $p \in M$ coincides with the full tangent space. As is well-known, the latter property implies {\it global connectivity by horizontal curves}, \cite{Chow,R} as 
well as {\it hypoellipticity} of the induced intrinsic sublaplacian in \cite{AGBR} (a.k.a. hypoelliptic Laplacian). The restriction of the standard Riemannian metric on $\mathbb{S}^7$ to 
$\mathcal{D}$ defines an inner product $g$ on the distributions and for each choice of $\mathcal{D}$ one obtains an SR manifold $(\mathbb{S}^7, \mathcal{D}, g)$.  

Trivializable SR structures on $\mathbb{S}^7$ were introduced in \cite{BFI} and we restate the definition. Let $A_1, \ldots, A_7 \in \mathfrak{so}(8)$ denote a family of skew-symmetric anti-commuting matrices (e.g. see \cite{BL,BLT} for a concrete realization): 
\begin{equation}\label{GL-Matrix-Clifford-relations}
A_iA_j+A_jA_i=- 2 \delta_{ij} \hspace{3ex} \mbox{\it and } \hspace{3ex} A_i^T=-A_i, \hspace{3ex} i,j=1, \ldots, 7. 
\end{equation}
To each of the matrices $A_i=(a_{lk}^{(i)})_{l,k=1}^8$ we assign a linear vector field on $\mathbb{R}^8$ by defining: 
\begin{equation}\label{canonical_vf}
X_{i}= X(A_{i})= \sum_{l,k=1}^8 a_{lk}^{(i)} x_k \frac{\partial}{\partial x_l}.
\end{equation}
Throughout the paper we consider the standard embedding $\mathbb{S}^7 \subset \mathbb{R}^{8}$. A simple calculation shows that the restriction of $X_1, \ldots, X_7$ to 
$\mathbb{S}^7$ are tangent to the sphere and orthogonal at any point $p \in \mathbb{S}^7$. In the following we call them {\it canonical vector fields}, see  \cite{Adams,James}. 
Hence $X_i$ in (\ref{canonical_vf}) define a trivialization of the tangent bundle $T\mathbb{S}^7$.
We define the distributions: 
\begin{equation}\label{notation-D-j}
\mathcal{D}_j:= \textup{span} \Big{\{} X_i \: : \: i=1,\ldots ,j\Big{\}}, \hspace{4ex} j=1,\ldots, 7. 
\end{equation} 
Note that $\mathcal{D}_7= T\mathbb{S}^7$ and in case of $j<7$ we have: 
\begin{thm}[\cite{BFI}]\label{Theorem_bracket_generating_condition_trivializable}
The distribution $\mathcal{D}_j$ is bracket generating if and only if $j \geq 4$. 
\end{thm}

Under an additional hypothesis on the distribution which is called {\it strongly bracket generating} or {\it fat} 
more can be said about the SR structure and their induced geodesic curves (cf. \cite{Mo, S1, S2}). In particular, in case of a fat distribution all geodesics are {\it normal}. Hence it is natural to ask whether $\mathcal{D}_j$, $j=4,5,6$ above are strongly bracket generating. First we recall the definition  (see \cite{S1}). 
An element $Y \in \mathcal{D}_q\subset T_q \mathbb{S}^7$ with $q \in \mathbb{S}^7$ is called a {\it 2-step bracket generator} if 
\begin{equation*}
T_q\mathbb{S}^7 = \mathcal{D}_q + \big{[} Y, \mathcal{D}_q\big{]}. 
\end{equation*}
Here $[Y, \mathcal{D}_q]$ denotes the vector subspace of $T_q\mathbb{S}^7$ spanned by $[\widetilde{Y}, X]_q$, where $X$ and $\widetilde{Y}$ are sections of 
$\mathcal{D}$ such that $\widetilde{Y}_q=Y$. If for all $q \in \mathbb{S}^7$ each non-trivial element $0 \ne Y \in \mathcal{D}_q$ is a 2-step bracket generator, then the distribution 
$\mathcal{D}$ is called {\it strongly bracket generating} or {\it fat}.  
\vspace{1ex}\par
First, we observe that in the case $\mathcal{D}= \mathcal{D}_j$, $(j=4,5,6)$ such a property does not depend on the choice of generators $X_{\ell}= X(A_{\ell})$ in 
(\ref{notation-D-j}) since any set of generators in (\ref{GL-Matrix-Clifford-relations}) is element-wise conjugate via a single orthogonal matrix to a fixed collection of such $A_j$. 
Hence we choose $A_1, \ldots, A_7\in \mathfrak{so}(8)$ explicitly as in \cite{BL}.  
In particular, with the identification $\mathbb{R}^8 \cong \mathbb{H} \times \mathbb{H}$ where
$$\mathbb{H}=\{ \alpha_0+ \alpha_1 {\bf i}+ \alpha_2 {\bf j}+ \alpha_3 {\bf k}\: : \: \alpha_j \in \mathbb{R} \}$$ 
denotes the quaternion numbers (${\bf i}^2= {\bf j}^2={\bf k}^2=-1$ and ${\bf ij}={\bf k}$ and {\bf i}, {\bf j}, {\bf k} anti-commute) we put: 
\begin{align}\label{Definition-A-j-for-j-5-7-section-1}
A_3&= 
\left( 
\begin{array}{cc}
0& B_3 \\ 
B_3 & 0
\end{array}
\right), \hspace{2ex} 
A_4= 
\left( 
\begin{array}{cc}
0 & 1 \\ 
-1 & 0
\end{array}
\right), \hspace{2ex} 
A_5:= 
\left( 
\begin{array}{cc}
{\bf i} & 0\\ 
0 & -{\bf i} 
\end{array}
\right), \\
A_6&=
\left( 
\begin{array}{cc}
{\bf j} & 0\\ 
0 & -{\bf j} 
\end{array}
\right), 
 \hspace{2ex} 
A_7=
\left( 
\begin{array}{cc}
{\bf k} & 0\\ 
0 & -{\bf k} 
\end{array}
\right)
\notag
\end{align}
acting on $\mathbb{R}^8 = \mathbb{H} \times \mathbb{H}$ via multiplication from the left. 
Recall that $B_3$ is the real $4 \times 4$ skew symmetric matrix (cf. \cite[Lemma 4.3]{BL}): 
\begin{equation*}
B_3= 
\left(
\begin{array}{cccc}
0 & 1 & 0 & 0\\
-1 & 0 & 0 & 0\\
0 & 0 & 0 & -1 \\
0 & 0 & 1 & 0
\end{array}
\right) \in \mathfrak{so}(4). 
\end{equation*}
The explicit form of $A_1$ and $A_2$ such that (\ref{GL-Matrix-Clifford-relations}) holds can be found in \cite{BL} but is not relevant for the analysis in the present paper. Now, we can prove a refined version of Theorem \ref{Theorem_bracket_generating_condition_trivializable}: 
\begin{prop}
The trivializable distributions $\mathcal{D}_j$, $(j=4,5,6)$ on $\mathbb{S}^7$ are strongly bracket generating if and only if $j=6$. 
\end{prop} 
\begin{proof}
According to our remark above it is sufficient to prove the theorem for the distributions $\widetilde{\mathcal{D}}_j:=\mathrm{span}  \{X(A_{8-j}), \ldots X(A_7) \}$ with $j=4,5,6$ instead of $\mathcal{D}_j$. 
\vspace{1mm}\\
{\bf Case} $j=6$: We will see below that $\widetilde{\mathcal{D}}_6$ is a contact distribution and therefore is known to be strongly bracket generating. Here we give another elementary 
proof. 
\vspace{1ex}\par 
Let $q \in \mathbb{S}^7\subset \mathbb{R}^8$ and choose $Y \in (\widetilde{\mathcal{D}}_6)_q$. 
Then $Y=Cq$, which can be identified with $X(C)_q$, where $C\in \mathfrak{so}(8)$ can be decomposed in the form $C=  \sum_{\ell=2}^7 \alpha_{\ell} A_{\ell}$ for some $\alpha_{\ell} \in \mathbb{R}$. 
With our previous notation we choose a section $\widetilde{Y}= X(C)$ of $\widetilde{\mathcal{D}}_6$ such that $\widetilde{Y}_q=Y$.  Note that $[X(C), X(A_{\ell})]=-X([C,A_{\ell}])$ and therefore we can identify: 
\begin{equation*}
\big{(}\widetilde{\mathcal{D}}_6\big{)}_q+\big{[}Y, (\widetilde{\mathcal{D}}_6)_q \big{]}\cong  \textup{span} \big{\{} A_{\ell}q, [C,A_{\ell}]q \: : \: 
\ell=2, \ldots ,7\big{\}}=:W_q \subset \mathbb{R}^8. 
\end{equation*}\par 
Let $v \in T_q\mathbb{S}^7$ be orthogonal to $W_q$ and note that for $\ell=2, \ldots, 7$: 
$$0=\langle v, [C, A_{\ell}]q\rangle =-2\langle Cv, A_{\ell}q\rangle.$$ 
Hence $\textup{span} \{v, Cv\}$ is in the orthogonal complement of $(\widetilde{\mathcal{D}}_6)_q$. Since $v$ and $Cv$ are orthogonal as well, we conclude that $v=0$, since otherwise we would have  constructed a 2-dimensional subspace inside  $\textup{span}\{A_1q\}= (\widetilde{\mathcal{D}}_6)_q^{\perp}$. 
In conclusion, $W_q= T_q\mathbb{S}^7$ for all $q \in \mathbb{S}^7$ and $\widetilde{\mathcal{D}_6}$ is strongly bracket generating. 
\vspace{1ex}\\
{\bf Case} $j=5$: Consider $q:=(1,0) \in \mathbb{H} \times \mathbb{H}$, the {\it north pole} of $\mathbb{S}^7$. We choose a vector 
$0 \ne Y=A_5q \in (\widetilde{\mathcal{D}}_5)_q$ and extend  it to a section $\widetilde{Y}:= X(A_5)$ of $\widetilde{\mathcal{D}}_5$. Then 
\begin{equation*}
\big{(} \widetilde{\mathcal{D}}_5\big{)}_q+ \big{[} Y, (\widetilde{\mathcal{D}}_5)_q \big{]}= \textup{span} \big{\{} A_{\ell}q, A_5A_rq \: : \: \ell=3,4,5,6,7; \: r= 3,4,6,7\big{\}}. 
\end{equation*}
A direct calculation using the explicit matrix representations in (\ref{Definition-A-j-for-j-5-7-section-1}) shows that the latter space is five dimensional and therefore 
cannot coincide with $T_q\mathbb{S}^7$. Hence $\widetilde{D}_4$ cannot be strongly bracket generating. 
\vspace{1ex}\\
{\bf Case} {\it $j=4$}: Consider again $q:=(1,0) \in \mathbb{H} \times \mathbb{H}$ and let $0 \ne Y=A_5q \in (\widetilde{\mathcal{D}}_4)_q$. Then we have: 
\begin{equation*}
\big{(} \widetilde{\mathcal{D}}_4\big{)}_q+ \big{[} Y, (\widetilde{\mathcal{D}}_4)_q \big{]}= \textup{span} \big{\{} A_{\ell}q, A_5A_rq \: : \: \ell=4,5,6,7; \: r= 4,6,7\big{\}}. 
\end{equation*}
A direct calculation using (\ref{Definition-A-j-for-j-5-7-section-1}) shows that this space is five dimensional and therefore a proper subspace of $T_q\mathbb{S}^7$. Hence 
$\widetilde{\mathcal{D}}_4$ cannot be strongly bracket generating. 
\end{proof}

By restricting the standard Riemannian metric of $\mathbb{S}^7$ to $\mathcal{D}_j$ we obtain three SR structures on $\mathbb{S}^7$ (different from the 
standard Riemannian structure) of ranks $j=4,5,6$, namely: 
\begin{equation*}
\mathbb{S}^7_{T,j}:=\big{(} \mathbb{S}^7, \mathcal{D}_j, g \big{)}, \hspace{3ex} j=4,5,6.
\end{equation*}
We call them {\it trivializable} since the distribution $\mathcal{D}_j$ is trivial as a vector bundle. To each of the manifolds $\mathbb{S}^7_{T,j}$ for $j=4,5,6$ we assign a {\it Hamiltonian in SR geometry} $$H_{\textup{sub}, j}: T^* \mathbb{S}^7 \rightarrow \mathbb{R}.$$  
Conversely,   $H_{\textup{sub},j}$ determines the corresponding SR structure on $\mathbb{S}^7$ uniquely (see \cite[Proposition 1.10]{Mo}). To define $H_{\textup{sub},j}$ we consider the symmetric bundle map 
\begin{equation}\label{definition-beta-j}
\beta^{(j)}: T^* \mathbb{S}^7 \rightarrow T\mathbb{S}^7, \hspace{4ex} j=4,5,6
\end{equation}
uniquely defined by the conditions: 
\begin{itemize}
\item[(a)] $\textup{im}(\beta^{(j)}_q)= (\mathcal{D}_j)_q$ where $q \in \mathbb{S}^7$, 
\item[(b)] $p(v)= g(\beta^{(j)}_q(p), v)$ for all $v \in (\mathcal{D}_j)_q$ and $p \in T_q^* \mathbb{S}^7$. 
\end{itemize}
Let $[X_j^*(q) \: : \: j=1, \ldots, 7]$ denote the dual basis to $[X_j(q) \: : \: j=1, \ldots, 7]$. Note that: 
\begin{equation*}
\beta^{(j)}_q(p)= \sum_{i=1}^jp_i X_i(q) \hspace{3ex} \mbox{\it where} \hspace{3ex} p := \sum_{i=1}^7 p_i X_i^*(q) \in T_q^*\mathbb{S}^7. 
\end{equation*}
Then  $H_{\textup{sub},j}$ for $j=4,5,6$ is defined as the fibrewise quadratic function: 
\begin{equation*}
H_{\textup{sub},j}(q,p):=\frac{1}{2}\: g_q\big{(} \beta^{(j)}_q(p), \beta^{(j)}_q(p) \big{)}=\frac{1}{2} \sum_{i=1}^j p_i^2. 
\end{equation*}
\par
In the following we frequently use the identification between $X_i(q)$ and $A_iq$ and the one between $X_i^*(q)$ and $\langle A_iq, \cdot \rangle$ with $i=1,\ldots, 7$, where $q \in \mathbb{S}^7 \subset \mathbb{R}^8$ and $\langle \cdot, \cdot \rangle$ denotes the Euclidean inner product of the ambient space $\mathbb{R}^8$. 
Moreover, we 
consider the cotangent bundle $T^*\mathbb{S}^7$ as a submanifold of $\mathbb{R}^8 \times \mathbb{R}^8$ via: 
\begin{equation*}
T_q^* \mathbb{S}^7 \ni  p=\sum_{i=1}^7 p_i X_i^*(q) \longleftrightarrow \Big{(} q, \sum_{i=1}^7 p_i A_iq \Big{)} \in \mathbb{R}^8 \times \mathbb{R}^8 
\end{equation*}
and we shortly write $p=\sum_{i=1}^7 p_i A_iq$. 
Then $p_i=p(X_i(q))= \langle A_iq, p \rangle$ and the Hamiltonian in SR geometry is expressed as:  
\begin{equation}\label{H_j_sub_definition}
H_{\textup{sub},j}: T^* \mathbb{S}^7 \rightarrow \mathbb{R}:  H_{\textup{sub},j}(q, \xi):=\frac{1}{2} \sum_{i=1}^j \langle A_{i}q, \xi\rangle^2, \hspace{3ex} (j=4,5,6). 
\end{equation}
Henceforth we rather use the notation $(q,\xi)$ instead of $(q,p)$ and think of $(q, \xi)$ as an element in $\mathbb{R}^8 \times \mathbb{R}^8$. 
For $j=4,5,6$ the Hamiltonian $H_{\textup{sub},j}$ naturally extends to a polynomial on $\mathbb{R}^8 \times \mathbb{R}^8$. 
\vspace{1ex}\par 
Other constructions of SR structures on $\mathbb{S}^7$ are based on the Hopf or quaternionic Hopf fibration  and are more common in the literature, \cite{BL, BauWang-1, BauWang, BFI-0, PGreiner, MaMo12,MaMo11, Mo}. Consider the matrices in (\ref{GL-Matrix-Clifford-relations}) and define the compact groups: 
\begin{align*}
G_{\textup{H}}& 
\cong \big{\{} e^{t A_7}= \cos t I+\sin t A_7 \: : \: t \in \mathbb{R} \big{\}}\cong \mathbb{S}^1, \\
G_{\textup{QH}}&
\cong \big{\{} x_0 I + x_1 K+ x_2 A_6+x_3A_7 \: : \: x_j \in \mathbb{R}, \: x_0^2+ \ldots +x_3^2=1 \big{\}} \cong \textup{SU}(2) \cong \mathbb{S}^3, 
\end{align*}
where $K:= A_6A_7$. We may choose the concrete realization of $A_i$ in \cite{BLT,BL}, see (\ref{Definition-A-j-for-j-5-7-section-1}). An easy calculation shows that 
\begin{equation*}
G_{\textup{QH}}= \left\{ 
\left(
\begin{array}{cc}
h & 0 \\
0 & -{\bf i} h {\bf i}
\end{array}
\right)
\: : \: h= x_0+x_1 {\bf i}+ x_2{\bf j}+ x_3 {\bf k}, \; \; x_0^2+ \ldots +x_3^2=1 \right\} \cong \mathbb{S}^3. 
\end{equation*}
Both groups, $G \in \{G_{\textup{H}}, G_{\textup{QH}}\}$  define a free right action on $\mathbb{S}^7$ as follows: 
\begin{equation*}
\mathbb{S}^7 \times G \rightarrow \mathbb{S}^7: (q,g) \mapsto q \cdot g := g^{-1} (q), 
\end{equation*}
where $g^{-1}(q)$ denotes matrix multiplication by $g^{-1}$ on $\mathbb{R}^8$ (leaving $\mathbb{S}^7$ invariant). Consider the distributions $\mathcal{V}_{\textup{H}}$ and 
$\mathcal{V}_{\textup{QH}}$ tangent to the fibres of $\pi: \mathbb{S}^7 \rightarrow \mathbb{S}^7 / G$: 
\begin{equation*}
\mathcal{V}_{\textup{H}}(q)= \textup{span} \{ A_7q \} \hspace{3ex} \mbox{\it and } \hspace{3ex} \mathcal{V}_{\textup{QH}}(q)= \textup{span} \big{\{} A_6q, A_7q, Kq\big{\}}. 
\end{equation*}
 If $H$ is a vector space with inner product and $V\subset H$, then by $V^{\perp}$ we denote the orthogonal complement of $V$ in $H$. We say that the distributions 
\begin{equation*}
\mathcal{D}_{\textup{H}}:= \mathcal{V}_{\textup{H}}^{\perp} \hspace{3ex} \mbox{\it and }\hspace{3ex} \mathcal{D}_{\textup{QH}}:= \mathcal{V}_{\textup{QH}}^{\perp} 
\end{equation*}
are induced by the Hopf and quaternionic Hopf fibrations, respectively. Clearly, with our notation in (\ref{notation-D-j}) we have 
\begin{equation*}
\mathcal{D}_{\textup{H}}= \mathcal{D}_6 = \textup{span} \big{\{} X_1, \ldots, X_6\big{\}}. 
\end{equation*}
Moreover, $\mathcal{D}_{\textup{H}}$ and $\mathcal{D}_{\textup{QH}}$ are bracket generating and therefore define SR manifolds which we denote by 
$\mathbb{S}_{\textup{H}}^7= \mathbb{S}^7_{\textup{T},6}= (\mathbb{S}^7,\mathcal{D}_{\textup{H}}, g)$ and $\mathbb{S}^7_{\textup{QH}}=
(\mathbb{S}^7, \mathcal{D}_{\textup{QH}}, g)$, respectively. On the one hand, the SR manifolds $\mathbb{S}_{\textup{H}}^7$  and $\mathbb{S}^7_{\textup{T},6}$ coincide by 
definition. On the other hand, it is known that the rank four structures $\mathbb{S}^7_{\textup{QH}}$  and $\mathbb{S}^7_{\textup{T},4}$ are not even locally isometric around any 
point as SR manifolds (see \cite{BL}).  It is also know that the distribution $\mathcal{D}_{\textup{QH}}$ is strongly bracket generating (cf. \cite{Mo}). 
\vspace{1mm}\par 
Next, we calculate the Hamiltonian of  $\mathbb{S}^7_{\textup{QH}}$ in SR geometry: 
\begin{equation*}
H_{ \textup{sub},\textup{QH}}: T^*\mathbb{S}^7 \rightarrow \mathbb{R}. 
\end{equation*}
Consider $\beta^{\textup{QH}}: T^* \mathbb{S}^7 \rightarrow T\mathbb{S}^7$ 
 defined analogously to (\ref{definition-beta-j}). One easily checks that 
 \begin{equation*}
 \beta_q^{\textup{QH}}(p)= \sum_{j=1}^5p_jX_j(q)- \sum_{i=1}^5p_i \langle A_iq, Kq\rangle X(K), 
 \end{equation*}
 where $p= \sum_{i=1}^7 p_i X_i^*(q) \in T_q^*\mathbb{S}^7.$ With the notation in (\ref{H_j_sub_definition}) it follows that: 
 \begin{align}\label{Hamiltonian-QH-definition}
 H_{\textup{sub}, \textup{QH}}(q, \xi)
 &=\frac{1}{2}\sum_{j=1}^5\xi_j^2- 
 \frac{1}{2} \Big{\langle} \xi, A_6A_7q \Big{\rangle}^2\\
 &= \frac{1}{2}\|\xi\|^2 - \frac{1}{2} \Big{(} \langle \xi, A_6q\rangle^2+\langle \xi, A_7q\rangle^2+ \langle \xi,Kq\rangle^2\Big{)}. \notag
 \end{align}
 Here $\| \cdot \|$ denotes the Euclidean norm on $\mathbb{R}^8$. 
 
Each of the above SR structures on $\mathbb{S}^7$ are equiregular and induce a (positive) {\it intrinsic sublaplacian} (cf. \cite{AGBR,BR} for the definitions) which, 
based on the bracket generating property of the underlying distribution, is a geometrically defined hypoelliptic second order differential operator, \cite{Hoermander}. 
 In the final section of the paper we relate the Poisson commuting first integrals with a set of commuting second order differential operators on $\mathbb{S}^7$ including the sublaplacian. These sublaplacians are intrinsic in the sense of \cite{AGBR} in all cases except for the structure $\mathbb{S}_{\textup{T},4}^7$, where we omit the first order term. Geometrically this means that in the definition of the operator we replace  Popp's measure on $\mathbb{S}_{\textup{T},4}^7$ by the standard volume of $\mathbb{S}^7$ as a Riemannian manifold (see \cite{BR,BL,BFI} for details). With the obvious notations the list of sublaplacians is as follows: 
 \begin{align}\label{Definition_sub_laplacians}
 \Delta_{\textup{T},j}^{\textup{sub}}&=- \sum_{i=1}^j X(A_i)^2, \hspace{4ex} j=4,5,6\\
 \Delta_{\textup{QH}}^{\textup{sub}}&= \Delta_{\mathbb{S}^7}+ X(A_6)^2+ X(A_7)^2 + X(K)^2, \notag
 \end{align}
where $\Delta_{\mathbb{S}^7}= - \sum_{i=1}^7 X(A_i)^2$ denotes the standard (positive) Laplacian on $\mathbb{S}^7$. From an analytical point of view these operators have been considered 
 in \cite{B,BFI,BLT,BL,PGreiner,MaMo12,Verd}, where  their spectral theory and heat kernels have been discussed.  
 
 We conclude this section with further notations and a proposition which is useful throughout our calculations. 
The Lie algebra $\mathfrak{so}(8)$ is equipped with the bi-invariant\footnote{i.e. invariant under the adjoint action.} bilinear form (inner product) 
\begin{equation}\label{GL-B}
B\left( X, Y \right):=\mathrm{Tr}\left(X^{\mathrm{T}}Y\right)=-\mathrm{Tr}\left(XY\right), \hspace{4ex} X, Y\in\mathfrak{so}(8),
\end{equation}
where, ``$\mathrm{Tr}$''  denotes the matrix trace. For the later use, we state the following property of the matrices $A_i$, $i=1, \cdots, 7$; $A_jA_k$, $1\leq j<k\leq 7$. 
\begin{prop}\label{prop_onb}
The matrices $\dfrac{1}{2\sqrt{2}}A_i$, $i=1, \cdots, 7$; $\dfrac{1}{2\sqrt{2}}A_jA_k$, $1\leq j<k\leq 7$ form an orthonormal basis of $\mathfrak{so}(8)$ with respect to the bi-invariant metric $B\left(\cdot, \cdot\right)$. 
\end{prop}
\begin{proof}
The Clifford relations \eqref{GL-Matrix-Clifford-relations} imply
\begin{align*}
B(A_i,A_i)
&=
-\mathrm{Tr}(A_i^2)=\mathrm{Tr}(\mathsf{E}_8)=8, \\
B(A_jA_k,A_jA_k)
&=
-\mathrm{Tr}(A_jA_kA_jA_k)=\mathrm{Tr}(\mathsf{E}_8)=8, 
\end{align*}
for $1\leq i\leq 7$, $1\leq j<k\leq 7$, where $\mathsf{E}_8$ is the $8\times 8$ identity matrix. 
As $A_i$, $1\leq i\leq 7$; $A_jA_k$, $1\leq i<j\leq 7$, are skew-symmetric, we conclude from \eqref{GL-Matrix-Clifford-relations} that the only other possibly non-trivial inner products between the above matrices are 
\begin{align*}
B(A_i,A_jA_r)
&=-\mathrm{Tr}(A_iA_jA_r), \\
B(A_iA_j,A_{\ell} A_r)&=-\mathrm{Tr}(A_iA_jA_{\ell} A_r), 
\end{align*}
where  $i,j,\ell, r \in \{1, \cdots, 7\}$ are pairwise distinct. 
Note that $A_iA_jA_{\ell}A_r=\pm A_pA_qA_s$ are symmetric for pairwise distinct indices $i,j,\ell,r$ and $p,q,s$ such that 
$$\{i,j,\ell,r,p,q,s\}= \{1, \ldots, 7\}.$$ 
Moreover, $\left(A_iA_jA_{\ell}A_r\right)^2=\mathsf{E}_8$ and hence the eigenvalues of $A_iA_jA_{\ell}A_r$ are $\pm 1$. 
The multiplicity of both eigenvalues are four, since the multiplication by $A_i$ interchanges the corresponding eigenspaces of $A_iA_jA_{\ell}A_r$. 
In fact, for an eigenvector $v\in\mathbb{R}^8$ of $A_iA_jA_{\ell}A_r$ belonging to the eigenvalue $\pm 1$ we have $\left(A_iA_jA_{\ell}A_r\right)A_iv=\mp A_iv$, as $\left(A_iA_jA_{\ell}A_r\right)A_i=-A_i\left(A_iA_jA_{\ell}A_r\right)$ and 
therefore $\mathrm{Tr}\left(A_iA_jA_{\ell}A_r\right)=0$.
\end{proof}
\section{Non-periodic normal geodesics}
Explicit solutions to the SR geodesic equations on $T^* \mathbb{S}^7$ have been found in the case of the contact and quaternionic contact structures
 $\mathbb{S}_{\textup{H}}$ and $\mathbb{S}_{\textup{QH}}$, respectively. See \cite{HuRo08,MaMo12} for the details. 
Under the canonical  projection $\pi: T^*\mathbb{S}^7 \rightarrow \mathbb{S}^7$ these curves descend to the so-called {\it normal  SR geodesics}. 
Among such curves there are non-period normal geodesics. 
In the present section we construct non-periodic normal geodesic in the case of the rank 4 and 5 trivializable SR structures $\mathbb{S}_{\textup{T},i}$, $i=4,5$. 
\medskip 
{
\par 
First, we express Hamilton's equation induced by $H_{\textup{sub},5}$ in global coordinates on $\mathbb{R}^8 \times \mathbb{R}^8$ 
 and we use the explicit realization of the anti-commuting matrices $A_1,\ldots , A_7 \in \mathfrak{so}(8)$ in 
\cite{BLT,BL}. However, only the concrete forms of $A_5$, $A_6$, $A_7$ in (\ref{Definition-A-j-for-j-5-7-section-1}) are relevant in our calculations. 
\vspace{1ex}\par 
We first consider the geodesic flow of the rank $5$ trivializable SR structure $\mathbb{S}_{\textup{T},5}^7$. 
Recall that Hamilton's equation have the form: 
\begin{equation}\label{HS_j=5}
\begin{cases}
\dot{q} &= \frac{\partial H_{\textup{sub},5}}{\partial \xi}= \sum_{\ell=1}^5 \langle A_{\ell}q, \xi \rangle A_{\ell} q, \\
\dot{\xi}&=- \frac{\partial H_{\textup{sub},5}}{\partial q} = - \sum_{\ell=1}^5 \langle q, A_{\ell} \xi \rangle A_{\ell} \xi. 
\end{cases}
\tag{HS-5}
\end{equation}
In the following we take initial values $(q_0, \xi_0)$ with $\xi_0 \in T_{q_0}\mathbb{S}^7$, i.e. $\langle q_0, \xi_0\rangle =0$. A solution $(q(t), \xi(t))$ to (\ref{HS_j=5}) with these initial data at time $t=0$ fulfills $\langle q(t), \xi(t) \rangle=0$ for all $t \in \mathbb{R}$. We write 
$$q= (q_1,q_2), \hspace{3ex} \xi=(\xi_1, \xi_2) \in \mathbb{R}^4 \times \mathbb{R}^4 
\cong \mathbb{H} \times \mathbb{H}.$$ 
The first component of (\ref{HS_j=5}) can be rewritten as follows: 
\begin{align*}
\dot{q}&=\sum_{\ell=1}^7 \langle A_{\ell}q, \xi \rangle A_{\ell}q
 - \langle A_6q, \xi \rangle A_6q - \langle A_7q, \xi \rangle A_7q \\
\hspace{3ex} &= \xi - \left(\langle {\bf j}q_1, \xi_1\rangle-\langle {\bf j}q_2, \xi_2\rangle\right)\begin{pmatrix}{\bf j}q_1 \\ -{\bf j}q_2\end{pmatrix}
 - \left(\langle {\bf k}q_1, \xi_1\rangle-\langle {\bf k}q_2, \xi_2\rangle\right)\begin{pmatrix}{\bf k}q_1 \\ -{\bf k}q_2\end{pmatrix} \\
&= 
\left( 
\begin{array}{c}
\xi_1 - \left(\langle {\bf j}q_1, \xi_1\rangle-\langle {\bf j}q_2, \xi_2\rangle\right){\bf j}q_1 -  \left(\langle {\bf k}q_1, \xi_1\rangle-\langle {\bf k}q_2, \xi_2\rangle\right) {\bf k}q_1 \\
\xi_2 + \left(\langle {\bf j}q_1, \xi_1\rangle-\langle {\bf j}q_2, \xi_2\rangle\right){\bf j}q_2 +  \left(\langle {\bf k}q_1, \xi_1\rangle-\langle {\bf k}q_2, \xi_2\rangle\right) {\bf k}q_2 
\end{array}
\right). 
\end{align*}
A similar calculation applies to the second component in (\ref{HS_j=5}). 
Note that $\|\xi\|$ is a first integral of the SR geodesic flow as we see in Section 4 and hence we may assume $\|\xi\|=1$. 
Then, we have 
\begin{align*}
\dot{\xi}&=-\sum_{\ell=1}^7 \langle q, A_{\ell}\xi \rangle A_{\ell}\xi
 + \langle q, A_6\xi \rangle A_6\xi + \langle q, A_7\xi \rangle A_7\xi \\
\hspace{3ex} &= -q + \left(\langle q_1,  {\bf j}\xi_1\rangle-\langle q_2,  {\bf j} \xi_2\rangle\right)\begin{pmatrix}{\bf j}\xi_1 \\ -{\bf j}\xi_2\end{pmatrix} 
 + \left(\langle q_1,  {\bf k}\xi_1\rangle-\langle q_2,  {\bf k}\xi_2\rangle\right)\begin{pmatrix}{\bf k}\xi_1 \\ -{\bf k}\xi_2\end{pmatrix} \\
&= 
\left( 
\begin{array}{c}
-q_1 + \left(\langle q_1, {\bf j}\xi_1\rangle-\langle q_2, {\bf j}\xi_2\rangle\right){\bf j}\xi_1 +  \left(\langle q_1, {\bf k}\xi_1\rangle-\langle q_2, {\bf k}\xi_2\rangle\right) {\bf k}\xi_1 \\
-q_2 - \left(\langle q_1, {\bf j}\xi_1\rangle-\langle q_2, {\bf j}\xi_2\rangle\right){\bf j}\xi_2 -  \left(\langle q_1,{\bf k} \xi_1\rangle-\langle q_2, {\bf k} \xi_2\rangle\right) {\bf k}\xi_2
\end{array}
\right). 
\end{align*}
\par 
Now, we present a non-periodic solution to (\ref{HS_j=5}). 
We look for solutions of the form $(q_1,q_2 ,\xi_1,\xi_2)=(q_1,q_2,{\bf k} q_1, {\bf k} q_2)$. 
Then the system of equations reduces to: 
\begin{equation}\label{reduced_eq}
\begin{cases}
\dot{q}_1&=(1- |q_1|^2+|q_2|^2) {\bf k} q_1, \\
\dot{q}_2 &=(1+|q_1|^2-|q_2|^2) {\bf k} q_2.
\end{cases}
\end{equation} 
Under the initial condition $(q_1(0), q_2(0))= (\tilde{q}_1, \tilde{q}_2) \in \mathbb{H} \times \mathbb{H}$ with $|\tilde{q}_1|^2+ |\tilde{q}_2|^2=1$ a solution is given by: 
\begin{align*}
q(t)=
\left( 
\begin{array}{c}
q_1(t)\\
q_2(t)
\end{array} 
\right)= 
\left( 
\begin{array}{c}
e^{(1- |\tilde{q}_1|^2+|\tilde{q}_2|^2)t {\bf k}} \tilde{q}_1 \\
e^{(1+ |\tilde{q}_1|^2-|\tilde{q}_2|^2)t {\bf k} } \tilde{q}_2
\end{array}
\right). 
\end{align*}
Note that $q_1(t)$ and $q_2(t)$ are periodic with (minimal) periods $\lambda_1=\dfrac{2\pi}{1- |\tilde{q}_1|^2+|\tilde{q}_2|^2}$ and $\lambda_2=\dfrac{2\pi}{1+ |\tilde{q}_1|^2-|\tilde{q}_2|^2}$, respectively. 
Choose now initial data $(\tilde{q}_1, \tilde{q}_2)$ such that 
\begin{equation*}
\frac{1- |\tilde{q}_1|^2+|\tilde{q}_2|^2}{1+ |\tilde{q}_1|^2-|\tilde{q}_2|^2} \in \mathbb{R} \setminus \mathbb{Q}
\end{equation*}
is irrational. 
Then, the components $q_1$ and $q_2$ of $q$ do not have a common period and therefore $q=(q_1,q_2)$ is non-periodic. Similarly, Hamilton's equation 
\begin{equation}\label{HS_j=4}
\begin{cases}
\dot{q} &= \frac{\partial H_{\textup{sub},4}}{\partial \xi}= \sum_{\ell=1}^4 \langle A_{\ell}q, \xi \rangle A_{\ell} q, \\
\dot{\xi}&=- \frac{\partial H_{\textup{sub},4}}{\partial q} = - \sum_{\ell=1}^4 \langle q, A_{\ell} \xi \rangle A_{\ell} \xi 
\end{cases}
\tag{HS-4}
\end{equation}
allows a non-periodic solution as follows: \\
In terms of the same coordinates $q=(q_1,q_2)$, $\xi=(\xi_1,\xi_2)$, the first component of \eqref{HS_j=4} can be written as 
\begin{align*}
\dot{q}&=\sum_{\ell=1}^7 \langle A_{\ell}q, \xi \rangle A_{\ell}q
 - \langle A_5q, \xi \rangle A_5q - \langle A_6q, \xi \rangle A_6q - \langle A_7q, \xi \rangle A_7q \\
&= 
\begin{pmatrix}
\xi_1 - \sum_{{\bf l}={\bf i}, {\bf j}, {\bf k}}\left(\langle {\bf l}q_1, \xi_1\rangle-\langle {\bf l}q_2, \xi_2\rangle\right){\bf l}q_1 \\
\xi_2 + \sum_{{\bf l}={\bf i}, {\bf j}, {\bf k}}\left(\langle {\bf l}q_1, \xi_1\rangle-\langle {\bf l}q_2, \xi_2\rangle\right){\bf l}q_2
\end{pmatrix}. 
\end{align*}
The second component of \eqref{HS_j=4} is written as 
\begin{align*}
\dot{\xi}&=-\sum_{\ell=1}^7 \langle q, A_{\ell}\xi \rangle A_{\ell}\xi
 + \langle q, A_5\xi \rangle A_5\xi + \langle q, A_6\xi \rangle A_6\xi + \langle q, A_7\xi \rangle A_7\xi \\
&= 
\begin{pmatrix}
-q_1 +  \sum_{{\bf l}={\bf i}, {\bf j}, {\bf k}}\left(\langle q_1, {\bf l}\xi_1\rangle-\langle q_2, {\bf l}\xi_2\rangle\right){\bf l}\xi_1 \\
-q_2 -  \sum_{{\bf l}={\bf i}, {\bf j}, {\bf k}}\left(\langle q_1, {\bf l}\xi_1\rangle-\langle q_2, {\bf l}\xi_2\rangle\right){\bf l}\xi_2
\end{pmatrix}. 
\end{align*}
\par 
Under the assumption $(q_1, q_2, \xi_1, \xi_2)=(q_1, q_2, {\bf k}q_1, {\bf k}q_2)$, these equations are reduced to the same equation \eqref{reduced_eq} as in the rank $5$ case, which admits a non-periodic solution as mentioned above. 
We summarize the above observation: 
\begin{prop}
The Hamiltonian systems (\ref{HS_j=4}) and (\ref{HS_j=5}) have non-periodic solutions. 
In particular, there are non-periodic SR normal geodesics on $\mathbb{S}_{\textup{T},4}^7$ and $\mathbb{S}_{\textup{T},5}^7$. 
\end{prop}
\section{Complete integrability of subriemannian geodesic flow}
\label{Section-4}
We first recall a method of how to construct families of first integrals on the tangent bundle of a Riemannian homogeneous space introduced by A. Thimm in \cite{T}. 
In the sequel these ideas will be applied to prove the complete integrability of the SR geodesic flow for all previously defined SR structures on $\mathbb{S}^7$. 
\medskip\par 
On a $2n$-dimensional symplectic manifold $\left(N,\omega\right)$, we consider a Hamiltonian system $\left(N,\omega, H\right)$ for a given Hamiltonian 
$H\in \mathcal{C}^{\infty}(N)$. 
Recall that $\left(N,\omega, H\right)$ is called \textit{completely integrable (in the sense of Liouville)} if there exist $n$ Poisson commuting and functionally independent first integrals $F_1, \ldots, F_{n-1}, F_n(=H)\in\mathcal{C}^{\infty}(N)$, including the Hamiltonian $H$. 
An important consequence of the complete integrability of a system is the {\it Liouville-Arnol'd Theorem} which states that in this case the mapping 
$$F=(F_1,\ldots, F_n)^T:N\rightarrow \mathbb{R}^n$$ admits a torus bundle structure around a compact and connected fibre over a regular value of $F$. Furthermore,  the flow can be realized as a linear flow on each torus appearing as a fibre of the bundle, cf.  \cite{arnold_1989}.  
The (normal) geodesic flow on a Riemannian, as well as on an SR manifold $M$, can be formulated as a Hamiltonian system on the cotangent bundle $N=T^{\ast}M$ equipped with the standard symplectic structure.  Recall that  $T^{\ast}M$ can canonically be  identified with the tangent bundle $TM$ when $M$ is equipped with a 
Riemannian metric. Under this condition, one can also consider the geodesic flow on the tangent bundle $TM$. 

\medskip
Although we are mostly concerned with the (co)tangent bundle of a manifold $M$ we continue with the setting of a general symplectic manifold $N$ which, however, admits a Hamiltonian group action by a Lie group $G$. By $\mathfrak{g}$ we denote the corresponding Lie algebra. 

Consider an equivariant momentum mapping $P:N\rightarrow \mathfrak{g}^{\ast}$ respective to the group action. 
For any ${\eta}\in\mathfrak{g}$, we denote the associated vector field on $N$ by ${\eta}_N$, i.e. 
$$\left.\dfrac{\mathsf{d}}{\mathsf{d}t}\right|_{t=0}F\left(e^{-t{\eta}}\cdot x\right)=:{\eta}_N\left[F\right]\left(x\right), 
\hspace{3ex} \forall x\in N, \hspace{3ex} \forall F\in\mathcal{C}^{\infty}(N).$$
Since the group action by $G$ is Hamiltonian, there is a function $F_{\eta}\in\mathcal{C}^{\infty}\left(N\right)$ whose Hamiltonian vector field coincides with 
${\eta}_N$, i.e. $X_{F_{\eta}}={\eta}_N$.  With this notation, we have $P\left(x\right)\left[{\eta}\right]=F_{\eta}\left(x\right)$, $\forall x\in N$, $\forall {\eta}\in\mathfrak{g}$, and 
$$\left\{F_{\eta}, F_{\eta}^{\prime}\right\}=F_{\left[{\eta},{\eta}^{\prime}\right]}, \hspace{3ex} \forall \hspace{1ex} {\eta}, {\eta}^{\prime}\in\mathfrak{g}.$$
{
The latter relation is extended to the $\mathbb{R}$-linear mapping $\mathcal{C}^{\infty}\left(\mathfrak{g}^{\ast}\right)\ni h\mapsto F_h=h\circ P= P^*h\in \mathcal{C}^{\infty}\left(N\right)$ compatible with Poisson structures:
\begin{equation*}
F_{\left\{h_1, h_2\right\}}=\left\{F_{h_1}, F_{h_2}\right\}, \qquad h_1, h_2\in\mathcal{C}^{\infty}\left(\mathfrak{g}^{\ast}\right), 
\end{equation*}
i.e. the momentum mapping $P$ is a Poisson mapping (see \cite{Gui-Stern}). 
}
As a consequence we see that any two Poisson commuting functions $f, g\in\mathcal{C}^{\infty}\left(\mathfrak{g}^{\ast}\right)$ give rise to Poisson commuting functions on $N$, $P^{\ast}f$, $P^{\ast}g$, respectively. 
In particular, if the symplectic manifold $N=T^{\ast}M$, is the cotangent bundle to a manifold $M$, then, for any ${\eta}\in \mathfrak{g}$, the function $F_{\eta}\in \mathcal{C}^{\infty}\left(T^{\ast}M\right)$ can  be taken as 
\begin{equation}\label{Lie_alg_induced_vf}
F_{\eta}\left(\alpha_q\right)=\alpha_q\left(\left(\eta_M\right)_q\right), \hspace{3ex} \forall q\in M, \hspace{2ex} \forall \alpha_q\in T_q^{\ast}M,
\end{equation}
where $\eta_M$ denotes the vector field on $M$ induced by $\eta$ through $\left(\eta_M\right)_q=\left.\dfrac{\mathsf{d}}{\mathsf{d}t}\right|_{t=0}e^{-t\eta}q$ for $q\in M$. 
See e.g. \cite[Chapter 9]{marsden-ratiu}. 
Using the above Poisson property of the momentum mapping, A. Thimm in \cite{T} gives a systematic method to construct Poisson commuting functions on $TM$. 
Assume that the Lie algebra $\mathfrak{g}$ is equipped with a bi-invariant (i.e $\mathrm{ad}$-invariant) non-degenerate bilinear form $B$.  
As usual $\mathfrak{g}$ and its dual $\mathfrak{g}^{\ast}$ are identified through $\mathfrak{g}\ni {
\eta}\mapsto B\left({
\eta}, \cdot\right)\in \mathfrak{g}^{\ast}$. 
Recall that the Lie-Poisson bracket on $C^{\infty}(\mathfrak{g})$ is defined by:  
\begin{equation}\label{Lie-Poisson-bracket}
\left\{F, G\right\}\left({\eta}\right)=B\left({
\eta}, \left[\left(\mathsf{d}F\right)_{\eta},\left(\mathsf{d}G\right)_{\eta}\right]\right), \qquad F, G\in\mathcal{C}^{\infty}\left(\mathfrak{g}\right), \; {\eta}\in\mathfrak{g}. 
\end{equation}
\par Recall that a subalgebra $\mathfrak{a}\subset \mathfrak{g}$ is called {\it non-degenerate} if the restriction $B_{\mathfrak{a}}:=B|_{\mathfrak{a}\times\mathfrak{a}}$ is a non-degenerate 
bilinear form on $\mathfrak{a}$. The next proposition is a fundamental ingredient in Thimm's method: 
\begin{prop}[\cite{T}]\label{porp_thimm}
Let $\mathfrak{a}_1, \mathfrak{a}_2\subset \mathfrak{g}$ be two non-degenerate subalgebras with the orthogonal projections $\pi_i:\mathfrak{g}\rightarrow \mathfrak{a}_i$, $i=1,2$. 
For bi-invariant functions $h_i\in\mathcal{C}^{\infty}\left(\mathfrak{a}_i\right)$ on each subalgebra, we have 
\[
\left\{h_1\circ \pi_1, h_2\circ \pi_2\right\}=0, 
\]
provided that $\left[\mathfrak{a}_1, \mathfrak{a}_2\right]\subset \mathfrak{a}_2$. 
\end{prop}
Note that the condition $\left[\mathfrak{a}_1, \mathfrak{a}_2\right]\subset \mathfrak{a}_2$ is satisfied e.g. if $\mathfrak{a}_1\subset \mathfrak{a}_2$ or if $\left[\mathfrak{a}_1,\mathfrak{a}_2\right]=0$. 
For the proof of Proposition \ref{porp_thimm}  see \cite[Proposition 4.1]{T}. 
As a special case, a number of Poisson commuting functions on $\mathfrak{g}^{\ast}$ can be constructed by choosing a chain of non-degenerate subalgebras $\mathfrak{a}_1\subset\mathfrak{a}_2\subset\cdots\subset \mathfrak{a}_k\subset \mathfrak{g}$. In fact, the pull-backs of bi-invariant functions on each of $\mathfrak{a}_i$, $i=1, \ldots, k$, through the respective orthogonal projections $\pi_i:\mathfrak{g}\rightarrow \mathfrak{a}_i$ are Poisson commuting.  
Pull-back along the momentum mapping $P: N\rightarrow \mathfrak{g}^{\ast}\cong  \mathfrak{g}$ leads to a number of Poisson commuting functions on $N$.  
\medskip \par
We apply the above ideas to the various SR-Hamiltonian flows on $\mathbb{S}^7$ introduced in the previous section. In the following we identify
\begin{equation*}
T^*_q \mathbb{S}^7\cong T_q\mathbb{S}^7\cong \Big{\{} (q, \xi) \in \mathbb{R}^8 \times \mathbb{R}^8 \: | \: \langle q,\xi \rangle:= q^T \cdot \xi =0 \Big{\}}. 
\end{equation*}
Then the momentum mapping {$P: T^*\mathbb{S}^7 \cong T\mathbb{S}^7 \rightarrow \mathfrak{so}(8)^*$ 
has the form: 
\begin{equation}\label{moment_map}
P(q, \xi)(\eta)= {F}_{\eta}(q, \xi)=\langle \eta q, \xi \rangle, \hspace{3ex} \eta \in \mathfrak{so}(8). 
\end{equation}
The bi-invariant inner product $B$ in (\ref{GL-B}) induces an isomorphism $\mathfrak{so}(8)^*\cong  \mathfrak{so}(8)$ 
which leads to the identification: 
\begin{equation}\label{GL-identification-P}
P(q, \xi)= q \wedge \xi := \frac{1}{2} (\xi q^T-q \xi^T). 
\end{equation}
In fact, if $\eta \in \mathfrak{so}(8)$, then: 
\begin{multline}\label{eq_**}
B\left(\eta, q\wedge \xi\right)
=
-\mathrm{Tr}\left(\eta\left(q\wedge \xi\right)\right)\\
=
-\dfrac{1}{2}\mathrm{Tr}\left(\eta\left(\xi q^{\mathrm{T}}-q\xi^{\mathrm{T}}\right)\right)
=
\mathrm{Tr}\left(\xi^{\mathrm{T}}\left(\eta q\right)\right)
=
\langle \eta q, \xi \rangle=P(q, \xi)(\eta).
\end{multline}
As in \cite{T} we consider {$C^{\infty}(\mathfrak{so}(8))$ with the Lie-Poisson bracket} (\ref{Lie-Poisson-bracket}). Let $\mathcal{S} \subset C^{\infty}(\mathfrak{so}(8))$ be a 
set of Poisson commuting functions and consider the {Poisson homomorphism}: 
\begin{equation*}
{P^*: C^{\infty}(\mathfrak{so}(8)) \rightarrow C^{\infty}(T\mathbb{S}^7): h \mapsto h \circ P. }
\end{equation*}
Then $P^*(\mathcal{S})$ defines a set of Poisson commuting functions in $C^{\infty}(T\mathbb{S}^7)$. 
\subsection{Complete integrability of SR geodesic flow of $\mathbb{S}^7_{\textup{T},5}$ and $\mathbb{S}^7_{\textup{QH}}$}
We aim to prove the complete integrability of the SR geodesic flow on the SR manifolds $\mathbb{S}^7_{\textup{T},5}$ and 
$\mathbb{S}^7_{\textup{QH}}$. 
Consider the following chain of strictly increasing non-degenerate subalgebras: 
\begin{equation}\label{chain_subalgebras_1}
\mathfrak{g}_1\subset \mathfrak{g}_2\subset \mathfrak{g}_3 \subset \mathfrak{g}_4 \subset \mathfrak{g}_5 \subset \mathfrak{g}_6 \subset \mathfrak{g}_7:=\mathfrak{so}(8),  
\end{equation}
where $\mathfrak{g}_i$ are defined as 
\begin{align*}
\mathfrak{g}_{\ell}
&=
\mathrm{span}_{\mathbb{R}}\left\{A_iA_j\mid 1\leq i<j\leq \ell+1\right\}, 
\quad \text{if}\; 1 \leq \ell \leq  4; \\
\mathfrak{g}_5
&=
\mathrm{span}_{\mathbb{R}}\left\{A_iA_j\mid 1\leq i<j\leq 5\, \text{or}\, (i,j)=(6,7)\right\}, \\
\mathfrak{g}_6
&=
\mathrm{span}_{\mathbb{R}}\left(\left\{A_iA_j\mid 1\leq i<j\leq 5\, \text{or}\, (i,j)=(6,7)\right\}\cup\left\{A_6,A_7\right\}\right). 
\end{align*}
In fact, a straightforward computation based on the Clifford relations \eqref{GL-Matrix-Clifford-relations} shows that $\mathfrak{g}_{\ell}$ are closed under the Lie bracket. 
Note that 
\[
\mathfrak{g}_7
=
\mathfrak{so}(8)
=
\mathrm{span}_{\mathbb{R}}\left(\left\{A_iA_j\mid 1\leq i <j\leq 7\right\}\cup\left\{A_k\mid 1\leq k\leq 7\right\}\right) 
\]
and that $\mathfrak{g}_{\ell}\cong \mathfrak{so}(\ell+1)$, $\ell =1,2,3,4$, $\mathfrak{g}_5\cong \mathfrak{so}(5)\oplus \mathfrak{so}(2)$, and $\mathfrak{g}_6\cong \mathfrak{so}(5)\oplus \mathfrak{so}(3)$. 
From the orthogonality of their generators 
(see Proposition \ref{prop_onb}) it follows that all algebras $\mathfrak{g}_i$ are non-degenerate in the sense of \cite{T}, i.e. the restriction 
$B_{\mathfrak{g}_i}$ of $B$ is non-degenerate on $\mathfrak{g}_i$ for $i=1, \ldots, 7$. 
}
On $\mathfrak{g}_i$ 
we consider the bi-invariant functions  
\[
f_i(X_i):=\dfrac{1}{2}B_{\mathfrak{g}_i}\left( X_i, X_i\right), \qquad X_i\in\mathfrak{g}_i. 
\]
With the orthogonal projections $\pi_i:\mathfrak{so}(8)\rightarrow \mathfrak{g}_i$ we obtain 
a set of Poisson commuting functions $f_i\circ \pi_i$, $i=1,\ldots, 7$ on $\mathfrak{so}(8)$ by Proposition \ref{porp_thimm}. 
\begin{prop}\label{prop-functional-independent-fi-pi}
The functions $\{ f_i\circ \pi_i\circ P \: :\: i=1, \ldots, 7\}$ on $T\mathbb{S}^7$ are Poisson commuting and functionally independent. 
\end{prop}
\begin{proof}
 According to our construction the compositions $ f_i\circ \pi_i\circ P$ for $i=1,\ldots, 7$ are Poisson commuting and it suffices to prove functionally independence. 
It is convenient to identify the tangent space $T_q\mathbb{S}^7$ at $q\in \mathbb{S}^7$ with the linear subspace 
\[
q\mathfrak{m}:=\left\{\xi q^{\mathrm{T}}-q\xi^{\mathrm{T}}\mid \xi\in\mathbb{R}^8\right\}\subset \mathfrak{so}(8). 
\]
This identification is obtained when we realize the sphere $\mathbb{S}^7$ as the homogeneous space $SO(8)/SO(7)$, where $SO(7)$ is regarded as the subgroup $\{1\}\times SO(7)$.  
By (\ref{GL-identification-P}) we have:
\[
f\circ P(q,\xi)=f\left(q\wedge \xi\right), \qquad (q,\xi)\in T_q\mathbb{S}^7, f\in C^{\infty}(\mathfrak{so}(8)). 
\]
For $X\in\mathfrak{so}(8)$, we have 
\[
\nabla \left(f_i\circ \pi_i\right)(X)
=
\nabla_{\mathfrak{g}_i}f_i \left(\pi_i(X)\right)
=
\pi_i(X)
\]
as $\pi_i$ is linear and $\nabla_{\mathfrak{g}_i}f_i(X_i)=X_i$, $\forall X_i\in\mathfrak{g}_i$. 
Here, $\nabla f$ is the gradient of $f\in\mathcal{C}^{\infty}\left(\mathfrak{so}(8)\right)$ with respect to the metric $B$ and $\nabla_{\mathfrak{g}_i} h$ is the gradient of $h\in\mathcal{C}^{\infty}\left(\mathfrak{g}_i\right)$ with respect to the metric $B_{\mathfrak{g}_i}$. 
Thus, for $q\wedge \xi\in q\mathfrak{m}$, $\nabla \left(f_i\circ \pi_i\right)\left(q\wedge \xi\right)$, $i=1, \ldots, 7$, are linearly independent if and only if $\pi_i\left(q\wedge \xi\right)$, $i=1, \ldots, 7$, are linearly independent. This condition is satisfied if for any $i=1, \ldots, 7$, 
\begin{equation}\label{eq_*}
\exists A\in \mathfrak{g}_{i-1}^{\perp_{B_{\mathfrak{g}_i}}}\; \text{\it such that}\; B\left( A, q\wedge \xi\right)\neq 0,
\end{equation}
under the convention $\mathfrak{g}_0=0$. 
Here $\mathfrak{g}_{i-1}^{\perp_{B_{\mathfrak{g}_i}}}$ denotes the orthogonal complement to $\mathfrak{g}_{i-1}$ in $\mathfrak{g}_i$ with respect to the inner product $B_{\mathfrak{g}_i}$. 
The proof that \eqref{eq_*} actually guarantees the linear independence of $\pi_i(q\wedge \xi)$, $i=1,\cdots, 7$, can be explained as follows: \\
Suppose that $$ \sum_{i=1}^7\alpha_i\pi_i(q\wedge \xi)=0.$$ By \eqref{eq_*}, 
there exists $Y_7\in \mathfrak{g}_{6}^{\perp_{B_{\mathfrak{g}_7}}}$ such that $B\left(Y_7, q\wedge \xi\right)\neq 0$. 
Then, we have 
\[
0
=
B\left(Y_7, \sum_{i=1}^7\alpha_i\pi_i(q\wedge \xi)\right)
=
\alpha_7 B\left(Y_7, \pi_7(q\wedge \xi)\right)
=
\alpha_7 B\left(Y_7, q\wedge \xi\right). 
\]
This means that $\alpha_7=0$ and $\displaystyle \sum_{i=1}^6\alpha_i\pi_i(q\wedge \xi)=0$. 
Inductively, $\alpha_1=\cdots=\alpha_7=0$ and hence $\pi_i(q\wedge \xi)$, $i=1,\cdots, 7$, are linearly independent. 

Note that \eqref{eq_*} is satisfied if there exists $A_i$ or $A_jA_k$ in $\mathfrak{g}_{i-1}^{\perp_{B_{\mathfrak{g}_i}}}$. 
In fact, if $A$ is either  $A_i$ or $A_jA_k$, then $Aq$ is a unit tangent vector in $T_q\mathbb{S}^7$ and hence by (\ref{eq_**}):
 $$B\left(A, q\wedge \xi\right)=\xi^{\mathrm{T}}\left(Aq\right)\neq 0$$ for almost all $\xi\in T_q\mathbb{S}^7$. 
For the above chain of non-degenerate subalgebras $\mathfrak{g}_i$, we can choose $A\in \mathfrak{g}_{i-1}^{\perp_{B_{\mathfrak{g}_i}}}$ in the form 
$A_i$ or $A_jA_k$ satisfying \eqref{eq_*} for each $i=1, \cdots ,7$. 
\end{proof}
Now, we show that the SR Hamiltonian $H_{\textup{sub},5}$ is in fact expressed as a linear combination of $f_ i\circ \pi_i\circ P$, $i=1, \cdots, 7$. 
{
As $A_iq$, $i=1, \cdots, 7$, form an orthonormal basis of $T_q\mathbb{S}^7$, we have for any $\xi\in T_q\mathbb{S}^7$ 
\begin{equation}\label{eq_!!}
\xi =\sum_{i=1}^7\left\langle A_iq, \xi\right\rangle A_iq  \hspace{3ex} \mbox{\it and} \hspace{3ex} \left\|\xi\right\|^2
=
\left\langle\xi, \xi\right\rangle
=
\displaystyle \sum_{i=1}^7\left\langle A_iq, \xi \right\rangle^2. 
\end{equation}

By \eqref{eq_**}, we have 
\begin{align}\label{eq_***}
2f_7\left(q\wedge \xi\right)
&=B(q \wedge \xi, q \wedge \xi)=
\sum_{1\leq i<j \leq 7}B( A_iA_j, q\wedge \xi)^2+\sum_{k=1}^7B( A_k, q\wedge \xi)^2 \notag \\
&=
\sum_{1\leq i<j \leq 7}\left\langle A_iA_jq, \xi\right\rangle^2+\sum_{k=1}^7\left\langle A_kq, \xi\right\rangle^2. 
\end{align}
To compute the first term, we consider 
\begin{align*}
\sum_{i,j=1, \cdots, 7}\left\langle A_iA_jq, \xi\right\rangle^2
&=
2\sum_{1\leq i<j \leq 7}\left\langle A_iA_jq, \xi\right\rangle^2+\sum_{k=1}^7\left\langle A_k^2q, \xi\right\rangle^2 \\
&=
2\sum_{1\leq i<j \leq 7}\left\langle A_iA_jq, \xi\right\rangle^2-7\underbrace{\left\langle q, \xi\right\rangle^2}_{=0}
=
2\sum_{1\leq i<j \leq 7}\left\langle A_iA_jq, \xi\right\rangle^2. 
\end{align*}
On the other hand, we have 
\begin{align*}
\sum_{i,j=1, \cdots, 7}\left\langle A_iA_jq, \xi\right\rangle^2
&=
\sum_{i,j=1, \cdots, 7}\left\langle A_jq, A_i \xi\right\rangle^2 
=
\sum_{i=1}^7\left\langle \sum_{j=1}^7\left\langle A_jq, A_i\xi\right\rangle A_jq, A_i\xi \right\rangle \\
&=
\sum_{i=1}^7\left\langle A_i\xi-\left\langle q, A_i\xi\right\rangle q, A_i\xi \right\rangle 
=
-\sum_{i=1}^7\left\langle A_i^2\xi, \xi\right\rangle -\sum_{i=1}^7\left\langle q, A_i\xi\right\rangle^2 \\
&\stackrel{\eqref{eq_!!}}{=}
7\left\|\xi\right\|^2-\left\|\xi\right\|^2
=
6\left\|\xi\right\|^2. 
\end{align*}
Therefore, $\displaystyle \sum_{1\leq i<j \leq 7}\left\langle A_iA_jq, \xi\right\rangle^2=3\left\|\xi\right\|^2$ and  \eqref{eq_***} is rewritten as 
\[
f_7\circ P\left(q, \xi\right)=f_7\left(q\wedge \xi\right)
=\frac{1}{2} \left(
3\left\|\xi\right\|^2+\left\|\xi\right\|^2\right)
=
2\left\|\xi\right\|^2.
\]
Since 
$2f_6 \circ \pi_6 \left(q\wedge \xi\right)-2f_5\circ \pi_5 \left(q\wedge \xi\right)=\left\langle A_6q, \xi\right\rangle^2+\left\langle A_7q, \xi\right\rangle^2$ we have: 
\begin{align*}
H_{\textup{sub},5}(q,\xi)
&=
\dfrac{1}{2}\sum_{k=1}^5\left\langle A_kq, \xi\right\rangle^2
=
\dfrac{1}{2}\sum_{k=1}^7\left\langle A_kq, \xi\right\rangle^2
-
\dfrac{1}{2}\left(\left\langle A_6q, \xi\right\rangle^2
+
\left\langle A_7q, \xi\right\rangle^2 \right)
\\
&=
\dfrac{1}{2}\left\|\xi\right\|^2
-
\dfrac{1}{2}\left(\left\langle A_6q, \xi\right\rangle^2
+
\left\langle A_7q, \xi\right\rangle^2\right)\\
&=
\dfrac{1}{4}f_7(q\wedge\xi)-f_6\circ \pi_6 (q\wedge \xi)+f_5\circ \pi_5(q\wedge \xi). 
\end{align*}
Hence, we have proved the decomposition of the SR Hamiltonian: 
\begin{equation}\label{eq_!!!}
H_{\textup{sub},5}=\dfrac{1}{4}f_7\circ P-f_6\circ \pi_6\circ P+f_5\circ \pi_5\circ P. 
\end{equation}
\begin{rem}
By taking suitable combinations of the functions $f_i\circ \pi_i \circ P$ in Proposition \ref{prop-functional-independent-fi-pi} we obtain the following 
list of functionally independent first integrals on $T\mathbb{S}^7$ 
which compared to  $f_i\circ \pi_i \circ P$ have less terms: 
\begin{align*}
F_{\ell}(q, \xi) &=f_{\ell+1} \circ \pi_{\ell+1} \circ P - f_{\ell} \circ \pi_{\ell} \circ P =
\sum_{i=1}^{\ell} \big{\langle} A_iA_{\ell+1}q, \xi \big{\rangle}^2, \hspace{4ex}\ell=1, \ldots, 4;\\
F_5(q, \xi)
&=\big{\langle} A_6A_7q, \xi \big{\rangle},\\
F_6(q,\xi)
&=\big{\langle} A_6q, \xi\big{\rangle}^2+ \big{\langle} A_7q, \xi\big{\rangle}^2,\\
F_7(q,\xi)&=H_{\textup{sub},5}(q, \xi). 
\end{align*}
Note that the first integral $F_7(q, \xi)$ above can be replaced by the $q$-independent function $\tilde{F}_7(q, \xi)= \|\xi\|^2= 2 H_{\textup{sub},5}(q, \xi)+ F_6(q, \xi)$. 
\end{rem}
The above sequence of subalgebras $\mathfrak{g}_i$, $i=1, \cdots, 7$, is also applicable to prove the complete integrability of the SR geodesic flow 
on $\mathbb{S}_{\textup{QH}}^7$, i.e.  with respect to the quaternionic contact structure. 
According to (\ref{Hamiltonian-QH-definition}) the Hamiltonian in this case has the form
\begin{equation*}
H_{\textup{sub}, \textup{QH}}(q,\xi)=\dfrac{1}{2}\left\|\xi\right\|^2-\dfrac{1}{2}\left(\left\langle A_6q, \xi\right\rangle^2
+\left\langle A_7q, \xi\right\rangle^2+\left\langle A_6A_7q, \xi\right\rangle^2\right).
\end{equation*} 
Applying the relation $2f_5\circ \pi_5(q\wedge \xi)-2f_4\circ \pi_4 (q\wedge \xi)=\left\langle A_6A_7q, \xi\right\rangle^2$ and \eqref{eq_!!!} shows: 
\begin{align*}
H_{\textup{sub}, \textup{QH}}
&=
H_{\textup{sub},5}-f_5\circ \pi_5 \circ P+f_4\circ \pi_4 \circ P\\
&=
\dfrac{1}{4}f_7-f_6\circ \pi_6\circ P+f_4\circ \pi_4 \circ P. 
\end{align*}
The above integrability result implies that the SR Hamiltonian system for $H_{\textup{sub}, \textup{QH}}$ is also completely integrable in the sense of Liouville. 
\subsection{Complete integrability of SR geodesic flow of $\mathbb{S}_{\textup{T},4}^7$ and $\mathbb{S}_{\textup{T},6}^7 = \mathbb{S}_{\textup{H}}^7$}
Consider the sequence of strictly increasing non-degenerate subalgebras
\begin{equation}\label{Lie-subalgebras-rank-4-case}
\mathfrak{h}_1\subset \mathfrak{h}_2\subset \mathfrak{h}_3 \subset \mathfrak{h}_4 \subset \mathfrak{h}_5 \subset \mathfrak{h}_6 \subset \mathfrak{h}_7:=\mathfrak{so}(8) 
\end{equation}
defined by 
\begin{align*}
\mathfrak{h}_{\ell}
&=
\mathrm{span}_{\mathbb{R}}\left\{A_iA_j\mid 1\leq i<j\leq \ell+1\right\}, 
\quad \text{if}\; 1 \leq \ell \leq  3; \\
\mathfrak{h}_4
&=
\mathrm{span}_{\mathbb{R}}\left\{A_iA_j\mid 1\leq i<j\leq 4\, \text{or}\, (i,j)=(5,6)\right\}, \\
\mathfrak{h}_5
&=
\mathrm{span}_{\mathbb{R}}\left\{A_iA_j\mid 1\leq i<j\leq 4\, \text{or}\, 5\leq i<j\leq7\right\}, \\
\mathfrak{h}_6
&=
\mathrm{span}_{\mathbb{R}}\left(\left\{A_iA_j\mid 1\leq i<j\leq 4\, \text{or}\, 5\leq i<j\leq7\right\}\cup\left\{A_5, A_6,A_7\right\}\right). 
\end{align*}
Note that $\mathfrak{h}_{\ell}\cong \mathfrak{so}(\ell+1)$, $\ell =1,2,3$, $\mathfrak{h}_4\cong \mathfrak{so}(4)\oplus \mathfrak{so}(2)$, $\mathfrak{g}_5\cong \mathfrak{so}(4)\oplus \mathfrak{so}(3)$, and $\mathfrak{h}_6\cong \mathfrak{so}(4)\oplus \mathfrak{so}(4)$. 

On $\mathfrak{h}_i$, $i=1, \ldots, 7$, we define the invariant quadratic functions 
\[
g_i(Y_i):=\dfrac{1}{2}B_{\mathfrak{h}_i}(Y_i,Y_i), \quad Y_i\in\mathfrak{h}_i. 
\]
}
As before, we can show that the compositions $g_i\circ \pi_{\mathfrak{h}_i}$, $i=1,\cdots, 7$, give  functionally independent and Poisson commuting functions, where $\pi_{\mathfrak{h}_i}:\mathfrak{so}(8)\rightarrow \mathfrak{h}_i$ denotes the orthogonal projection. Note that the SR Hamiltonian 
\begin{align*}
H_{\textup{sub},4}(q,\xi)
=
\dfrac{1}{2}\sum_{k=1}^4\left\langle A_kq, \xi\right\rangle^2
&=
\dfrac{1}{2}\sum_{k=1}^7\left\langle A_kq, \xi\right\rangle^2
-
\dfrac{1}{2}\sum_{k=5}^7\left\langle A_kq, \xi\right\rangle^2\\
&=
\dfrac{1}{2}\left\|\xi\right\|^2
-
\dfrac{1}{2}\sum_{k=5}^7\left\langle A_kq, \xi\right\rangle^2
\end{align*}
can be decomposed as 
\[
H_{\textup{sub},4}
=
\dfrac{1}{4}g_7-  g_6\circ \pi_{\mathfrak{h}_6}\circ P +g_5\circ \pi_{\mathfrak{h}_5}\circ P. 
\]
Thus, the SR Hamiltonian system for $H_{\textup{sub},4}$ is completely integrable, as well. 

\medskip

Finally, we consider the SR geodesic flow on $\mathbb{S}_{\textup{T},6}^7 = \mathbb{S}_{\textup{H}}^7$.
Take the sequence of strictly increasing non-degenerate subalgebras
\begin{equation}\label{Lie-subalgebras-rank-6-case}
\mathfrak{k}_1\subset \mathfrak{k}_2\subset \mathfrak{k}_3 \subset \mathfrak{k}_4 \subset \mathfrak{k}_5 \subset \mathfrak{k}_6 \subset \mathfrak{k}_7:=\mathfrak{so}(8) 
\end{equation}
defined as 
\begin{align*}
\mathfrak{k}_{\ell}
&=
\mathrm{span}_{\mathbb{R}}\left\{A_iA_j\mid 1\leq i<j\leq \ell+1\right\}, 
\quad \text{if}\; 1 \leq \ell \leq  5; \\
\mathfrak{k}_6
&=
\mathrm{span}_{\mathbb{R}}\left(\left\{A_iA_j\mid 1\leq i<j\leq 6\, \right\}\cup\left\{A_7\right\}\right). 
\end{align*}
Note that $\mathfrak{k}_{\ell}\cong \mathfrak{so}(\ell+1)$, $\ell =1,2,3,4,5$, and $\mathfrak{k}_6\cong \mathfrak{so}(6)\oplus \mathfrak{so}(2)$. Consider the invariant 
functions on each of $\mathfrak{k}_i$, $i=1,\cdots, 7$: \\
\[
h_i(Z_i):=\dfrac{1}{2}B_{\mathfrak{k}_i}(Z_i,Z_i), \quad  Z_i\in\mathfrak{k}_i. 
\]
Clearly, 
\[
H_{\textup{sub},6}(q,\xi)
=
\dfrac{1}{2}\sum_{k=1}^6\left\langle A_kq, \xi\right\rangle^2
=
\dfrac{1}{2}\sum_{k=1}^7\left\langle A_kq, \xi\right\rangle^2
-
\dfrac{1}{2}\left\langle A_7q, \xi\right\rangle^2
=
\dfrac{1}{2}\left\|\xi\right\|^2
-
\dfrac{1}{2}\left\langle A_7q, \xi\right\rangle^2
\]
and hence 
\[
H_{\textup{sub},6}
=
\dfrac{1}{4}h_7-  h_6\circ \pi_{\mathfrak{k}_6} \circ P+h_5\circ \pi_{\mathfrak{k}_5} \circ P. 
\]
The next theorem summarizes the above results. 
}
\begin{thm}
The SR geodesic flows on $\mathbb{S}_{\textup{T},j}^7$ for $j=4, \ldots, 5$, on $\mathbb{S}_{\textup{T},6}^7 = \mathbb{S}_{\textup{H}}^7$, and on $\mathbb{S}_{\textup{QH}}^7$ 
are completely integrable in the sense of Liouville. 
\end{thm}
\section{Subriemannian isometries}
We review the definition of Killing vector fields and their relation to the SR Hamiltonian function. 
See \cite[\S 5.3]{ABB} for the details. 
A (complete) vector field on an SR manifold is called a Killing vector field if it generates a one-parameter group of SR isometries. 
Given a vector field $X$ on an SR manifold $M=(M, \mathcal{H}, g)$ we define the corresponding smooth function $f_X$ on $T^{\ast}M$ through $f_X(\xi)=\xi(X_q)$ for any $\xi\in T^{\ast}_qM$, $q\in M$. 
(This notation should not be confused with the previous function $F_{\eta}$, $\eta\in\mathfrak{g}$, in \eqref{Lie_alg_induced_vf}.) 
Taking into account  the canonical vector field $X(A)$ for $A\in \mathfrak{so}(8)$ in \eqref{canonical_vf} and the function $F_A$ in \eqref{Lie_alg_induced_vf} for $A=\eta$, we have 
\[
f_{X(A)}=F_A. 
\]
Then, we have the following characterization of Killing vector fields: 
\begin{lem}[See Lemma 5.15, p.158, in \cite{ABB}]\label{lem_hamiltonian_killing}
A complete vector field $X$ on $M$ is a Killing vector field of the SR manifold $M$ if and only if 
\[
\left\{H_{\textup{sub}}, f_X\right\}=0,
\]
where $H_{\textup{sub}}$ denotes the SR Hamiltonian. 
\end{lem}

In what follows, we discuss the SR isometry group for each of the SR structures on $\mathbb{S}^7$ on the basis of Lemma \ref{lem_hamiltonian_killing}. 
Recall that a smooth mapping $\Phi: M\rightarrow M^{\prime}$ between two given SR manifolds $(M, \mathcal{H}, g)$ and $(M^{\prime}, \mathcal{H}^{\prime}, g^{\prime})$ is called an  {\it (infinitesimal)} {\it subriemannian (SR) isometry} if it is a diffeomorphism which is horizontal, i.e. the differential $\Phi_{\ast}=(d\Phi)_q: T_qM\rightarrow T_{\Phi(q)}M$ satisfies $\Phi_{\ast}\left(\mathcal{H}_q\right)\subset \mathcal{H}_{\Phi(q)}^{\prime}$, and isometric, i.e. $\Phi_{\ast}$ is an isometry between the metric spaces $\left(\mathcal{H}_q, g_q\right)$ and $\left(\mathcal{H}_{\Phi(q)}^{\prime}, g_{\Phi(q)}^{\prime}\right)$. 
The group of all the SR isometries for a given SR manifold is called the SR isometry group. 
See \cite{ABB,BR,BLT,CaLeDo,S1,S2} for the details. 
\begin{rem}\label{rem-isometry-group-metric}
We may as well consider $M$ as a metric space $(M, d_{\textup{CC}})$ with respect to the {\it Carnot-Carath\'{e}odory distance}. Let $x,y \in M$: 
\begin{equation*}
d_{\textup{CC}}(x,y):= \inf \Big{\{}\ell(\gamma)\: : \: \gamma: [a,b] \rightarrow M, \hspace{1ex} \gamma^{\prime}(t) \in \mathcal{H}_{\gamma(t)}, 
\hspace{1ex} \gamma(a)=x, \hspace{1ex} \gamma(b)=y \Big{\}}, 
\end{equation*}
where $\ell(\gamma)= \displaystyle \int_a^b \sqrt{g_{\gamma(t)}\left(\gamma^{\prime}(t), \gamma^{\prime}(t)\right)} \mathsf{d}t$ denotes the length functional in SR geometry (cf. \cite{Mo}). A 
distance preserving homeomorphism $\Phi: M \rightarrow M$, i.e. 
\begin{equation*}
d_{\textup{CC}}(x,y)= d_{\textup{CC}}\big{(} \Phi(x), \Phi(y) \big{)} \hspace{3ex} \forall \; x,y \in M
\end{equation*}
is called a {\it (metric SR) isometry} of $M$. Infinitesimal SR isometries are metric SR isometries. 
Conversely, on an equiregular SR manifold SR metric isometries are known 
to be smooth, \cite[Cor. 1.5]{CaLeDo} and define infinitesimal ismetries, \cite[Theorem 8.2, (c)]{S1}. Moreover, $d_{\textup{CC}}$ induces the manifold topology on $M$. 
\end{rem}
\vspace{1ex}\par 
In the present paper, we denote the SR isometry group of $\mathbb{S}_{\textup{T}, \ell}^7$, $\ell =4, 5$, $\mathbb{S}_{\textup{QH}}^7$, and $\mathbb{S}_{\textup{T}, 6}^7=\mathbb{S}_{\textup{H}}^7$, respectively, by $\mathrm{Iso}_{\textup{sub}}\left(\mathbb{S}_{\textup{T},\ell}^7\right)$, $\mathrm{Iso}_{\textup{sub}}\left(\mathbb{S}_{\textup{QH}}^7\right)$, and $\mathrm{Iso}_{\textup{sub}}\left(\mathbb{S}_{\textup{T}, 6}^7\right)=\mathrm{Iso}_{\textup{sub}}\left(\mathbb{S}_{\textup{H}}^7\right)$. 
In a canonical way, $O(8)$ acts on $\mathbb{R}^8$ and $\mathbb{S}^7$ via left-multiplication. 
The induced action on 
$T^*\mathbb{R}^8 \cong \mathbb{R}^8 \times \mathbb{R}^8$ is 
given by: 
\begin{equation*}
a\cdot ({q},\xi)= (a{q}, a\xi) \hspace{3ex} \mbox{\it where } \hspace{3ex} {(q,\xi)\in \mathbb{R}^8\times \mathbb{R}^8}, \quad a\in O(8). 
\end{equation*}
As the group $O(8)$ acts on $\mathbb{S}^7$ transitively and isometrically with respect to the standard Riemannian metric, we focus on the SR isometry group in relation to the $O(8)$-action. 

\begin{lem}\label{Lemma-intersection-iso-O-8}
Let $\textup{Iso}_{\textup{sub}}(\mathbb{S}^7)$ be one of the above SR isometry groups on $\mathbb{S}^7$. Then the intersection 
$\textup{Iso}_{\textup{sub}}(\mathbb{S}^7)\cap O(8)$ is a Lie subgroup of 
$O(8)$. 
\end{lem}
\begin{proof}
As is well-known, any closed subgroup of the general linear group is a Lie group (see \cite[p. 71-72]{MoZi} or \cite{JvN}). Hence, it is sufficient to show 
that $\textup{Iso}_{\textup{sub}}(\mathbb{S}^7)\cap O(8)$ is closed in $O(8)$. Recall that all the SR structures in Section \ref{section-2} are equiregular. 
Clearly, the pointwise limit $\Phi$ of a sequence $(\Phi_n)_{n\in \mathbb{N}}$ of (metric) isometries in $\textup{Iso}_{\textup{sub}}(\mathbb{S}^7)\cap O(8)$ is a metric isometry and also contained in $O(8)$. 
\end{proof}
\subsection{SR isometry group of $\mathbb{S}_{\textup{T},5}^7$}
With $t\in \mathbb{R}$ and $1\leq i <j\leq 7$, we use the notation 
\begin{equation*}
g_{ij}(t):= e^{tA_iA_j} \in O(8). 
\end{equation*}
\begin{lem}\label{Lemma_invariance_Hamiltonian}
The SR Hamiltonian $H_{\textup{sub},5}$ is invariant under the action induced by $g_{ij}(t)$ where $1\leq i< j\leq 5$ or $(i,j)=(6,7)$ for all $t \in \mathbb{R}$, i.e.
\begin{equation*}
H_{\textup{sub},5}\big{(}g_{ij}(t){q},g_{ij}(t)\xi\big{)}= H_{\textup{sub},5}({q},\xi) \hspace{3ex} \mbox{\it for all} \hspace{3ex} ({q},\xi) \in \mathbb{R}^8 \times \mathbb{R}^8. 
\end{equation*}
\end{lem}
\begin{proof}
We only check the case $(i,j)=(1,2)$ and the other cases can be proved similarly. Since $g_{12}(t)$ commutes with $A_3, \ldots, A_5$ it suffices to show that: 
\begin{equation}\label{Lemma_1_sufficient_condition}
\big{\langle} A_1g_{12}(t)q, g_{12}(t)\xi \big{\rangle}^2+ \big{\langle} A_2g_{12}(t){q}, g_{12}(t)\xi\big{\rangle}^2= \big{\langle} A_1{q}, \xi \big{\rangle}^2+ \big{\langle} A_2{q}, \xi\big{\rangle}^2. 
\end{equation}
Note that the anti-commutation relation (\ref{GL-Matrix-Clifford-relations}) imply $g_{12}(t)= { \mathsf{E}_8}\cos t+ A_1A_2 \sin t$. 
Inserting this relation above shows the statement by a direct calculation. 
\end{proof}
Now we consider the Lie subalgebra of $\mathfrak{g}_5 \subset \mathfrak{so}(8)$ defined in (\ref{chain_subalgebras_1}): 
\begin{equation*}
\mathfrak{g}_5=\textup{Lie} \big{\{} \left. A_iA_j, A_6A_7\: \right| \: i,j=1, \ldots 5, \; \; i<j \big{\}}.
\end{equation*}
Note that the above set of generators of $\mathfrak{g}_5$ is invariant up to a sign under the Lie bracket (commutator). Hence the Lie hull coincides with the 
linear span of generators.  

For each subalgebras $\mathfrak{g}_{\ell}$ in \eqref{chain_subalgebras_1}, $\ell=1, \cdots, 6$, let $G_{\ell} \subset O(8)$ be the corresponding Lie group and $\exp: \mathfrak{g}_{\ell} \rightarrow G_{\ell}$ the (matrix) exponential. 
Recall that the exponential mapping is surjective for an arbitrary compact connected group cf. \cite[Corollary 4.48]{knapp_2002}. 
\begin{lem}\label{Lemma_invariance_G_ell}
Each group $G_{\ell}$ for $\ell=1, \ldots, 5$ is generated by the set: 
$$\big{\{} \exp(tA_iA_j) \: \mid \: A_iA_j \: \: \textup{\it generator of } \: \mathfrak{g}_{\ell}, \: t \in \mathbb{R}\big{\}}.$$
In particular, the induced action by {$G_{\ell}$} leaves the SR Hamiltonian $H_{\textup{sub},5}$ invariant. 
\end{lem}
\begin{proof}
Inductively, the statement follows from the Trotter product formula: 
\begin{equation*}
\exp\{A+A^{\prime}\}= \lim_{n \rightarrow \infty} \Big{(} \exp\Big{\{} \frac{A}{n}\Big{\}} \exp \Big{\{} \frac{A^{\prime}}{n}\Big{\}} \Big{)}^n, \hspace{4ex} A,A^{\prime}
\in \mathbb{R}^{8 \times 8}, 
\end{equation*}
together with Lemma \ref{Lemma_invariance_Hamiltonian}. 
\end{proof}
\begin{cor}\label{cor_g_5}
With our previous notation we obtain a chain of subgroups: 
\begin{equation*}
G_1 \subset G_2 \subset G_3 \subset G_4 \subset G_5 \subset \textup{Iso}_{\textup{sub}}(\mathbb{S}_{\textup{T},5}^7). 
\end{equation*}
In particular, $\dim \textup{Iso}_{\textup{sub}}(\mathbb{S}_{\textup{T},5}^7)\geq \dim G_5=11$. 
The group $G_5$ and (hence $\textup{Iso}_{\textup{sub}}(\mathbb{S}_{\textup{T},5}^7)$ as well) act transitively on $\mathbb{S}^7$. 
\end{cor}
\begin{proof}
We show that $G_5$ is a subgroup of $\textup{Iso}_{\textup{sub}}(\mathbb{S}_{\textup{T},5}^7)$. Since $A_6A_7$ commutes with $A_{\ell}$ where $\ell=1, \ldots, 5$ we conclude that {the differential $\left(de^{tA_6A_7}\right)_q=e^{tA_6A_7}: T_q\mathbb{S}^7\rightarrow T_{e^{tA_6A_7}q}\mathbb{S}^7$} maps $\mathcal{H}_q$ to $\mathcal{H}_{e^{tA_6A_7}q}$ for all $t \in \mathbb{R}$. Similarly, $e^{tA_iA_j}$ for $i,j=1,\ldots, 5$  acts isometrically between 
vertical spaces: 
\begin{equation*}
e^{tA_iA_j}: V_q= \textup{span} \big{\{} A_6q,A_7q\big{\}}= \mathcal{H}_q^{\perp} \rightarrow V_{e^{tA_iA_j}q} = \mathcal{H}_{e^{tA_iA_j}q}^{\perp} 
\end{equation*}
and therefore $e^{tA_iA_j}$ maps $\mathcal{H}_q$ to $\mathcal{H}_{e^{tA_iA_j}q}$, as well. According to Lemma \ref{Lemma_invariance_G_ell} this property generalizes to all 
elements in $G_5$. Let $q \in \mathbb{S}^7$, $a \in G_5$ and $\xi \in \mathcal{H}_q$, then: 
\begin{equation*}
g_{aq} \left(\mathsf{d}a_q \xi, \mathsf{d}a_q \xi \right)= 2H_{\textup{sub},5}(aq,a\xi)= 2H_{\textup{sub},5}(q,\xi)= g_q( \xi,\xi ). 
\end{equation*} 
The last statement is Corollay 3.6. in \cite{BLT}. 
\end{proof}
The following lemma is useful in subsequent computations. 
\begin{lem}\label{com_rel_ham_f_a}
For $A, B_1, \ldots, B_m\in \mathfrak{so}(8)$, $m\in\mathbb{N}$, we have 
\[
\left\{H,F_{A}\right\}=\sum_{\ell=1}^mF_{B_{\ell}}F_{[B_{\ell},A]}, 
\]
where $H:=\displaystyle\dfrac{1}{2}\sum_{\ell=1}^mF_{B_{\ell}}^2$. 
\end{lem}
\begin{proof}
By a straightforward computation, we have 
\[
\left\{H, F_A\right\}
=
\dfrac{1}{2}\sum_{\ell=1}^5\left\{F_{B_{\ell}}^2, F_A\right\}
=
\sum_{\ell=1}^5F_{B_{\ell}}\left\{F_{B_{\ell}}, F_A\right\}
=
\sum_{\ell=1}^5F_{B_{\ell}}F_{\left[B_{\ell}, A\right]}.
\]
\end{proof}
We further show the following characterization of the infinitesimal isometries in the orthogonal group. 
\begin{thm}
The Lie algebra of $\mathrm{Iso}_{\textup{sub}}\left(\mathbb{S}_{\textup{T},5}^7\right)\cap O(8)$ coincides with $\mathfrak{g}_5$ in \eqref{chain_subalgebras_1}. 
\end{thm}
\begin{proof}
For the proof, we show the following: 
\begin{enumerate}
\item If $A\in \mathfrak{g}_5$, then $\left\{H_{\textup{sub},5}, F_A\right\}=0$. 
\item If $A\in \mathfrak{g}_5^{\perp}\subset \mathfrak{so}(8)$ and $\left\{H_{\textup{sub},5}, F_A\right\}=0$, then $A=0$. 
\end{enumerate}
\noindent (1)\; 
Let \begin{equation}\label{infinitesimal_isometry}
A
=
\sum_{1\leq j<k\leq 5}\beta_{jk}A_jA_k+\beta_{67}A_6A_7, 
\quad \beta_{jk}\in \mathbb{R}, 
\end{equation}
in $\mathfrak{g}_5$. 
Then, by Lemma \ref{com_rel_ham_f_a}, we have 
\[
\left\{H_{\textup{sub},5}, F_A\right\}
=
\sum_{\ell=1}^5f_{A_{\ell}}f_{\left[A_{\ell}, A\right]}
=
-2f_{A_j}f_{A_k}+2f_{A_k}f_{A_j}
=
0.
\]
Here we have used $\left[A_j, A_jA_k\right]=-2A_k$, $\left[A_k, A_jA_k\right]=2A_j$ for $1\leq j<k\leq 5$. 
Therefore, by Lemma \ref{lem_hamiltonian_killing}, the vector field $X(A)$ for $A$ as in \eqref{infinitesimal_isometry} is a Killing vector field. 

\noindent (2)\; 
Suppose that $$A=\displaystyle \sum_{i=1}^7\alpha_iA_i+\sum_{j=1}^5\left(\beta_{j6}A_jA_6+\beta_{j7}A_jA_7\right)\in\mathfrak{g}_5^{\perp}$$ satisfies $\left\{H_{\textup{sub},5}, F_A\right\}=0$. 
Since the function $\|\xi\|^2$ Poisson commutes with $F_A$ and $F_{A_6}^2+F_{A_7}^2=\|\xi\|^2-2H_{\textup{sub},5}$, the condition $\left\{H_{\textup{sub},5}, F_A\right\}=0$ is equivalent to $F_{A_6}F_{\left[A_6,A\right]}+F_{A_7}F_{\left[A_7,A\right]}=0$. 
By 
\begin{align*}
\left[A_\ell, A\right]
=
-2\sum_{i=1,\ldots,7,i\neq \ell}\alpha_iA_iA_{\ell}+2\sum_{j=1}^5\beta_{j\ell}A_j, 
\end{align*}
$\ell=6,7$, we have, for $\forall (q,\xi)\in T\mathbb{S}^7\subset \mathbb{R}^8\times \mathbb{R}^8$, 
\begin{align}\label{condition_reduced_1}
&\left\langle A_6q, \xi\right\rangle \cdot \left(\sum_{i=1,\ldots, 7, i\neq 6}\alpha_i\left\langle A_iA_6q, \xi\right\rangle-\sum_{j=1}^5\beta_{j6}\left\langle A_jq,\xi\right\rangle\right) \notag \\
&\qquad +
\left\langle A_7q, \xi\right\rangle \cdot \left(\sum_{i=1}^6\alpha_i\left\langle A_iA_7q, \xi\right\rangle-\sum_{j=1}^5\beta_{j7}\left\langle A_jq,\xi\right\rangle\right)
=
0. 
\end{align}
Assume now that $\xi=A_{\ell}A_6q$ for a fixed $\ell=1, \ldots, 5$. 
Then, the first term in \eqref{condition_reduced_1} vanishes and we have 
\begin{align*}
\left\langle q, A_{\ell}A_6A_7q\right\rangle \cdot \left(\sum_{i=1, \ldots, 5, i\neq\ell}\alpha_i\left\langle q, A_iA_{\ell}A_6A_7q\right\rangle-\sum_{j=1, \ldots, 5, j\neq\ell}\beta_{j7}\left\langle q, A_jA_{\ell}A_6q\right\rangle\right)
=
0,
\end{align*}
as $\left\langle q, A_{\ell}^2A_6A_7q\right\rangle= -\left\langle q, A_6A_7q\right\rangle=0$, $\left\langle q, A_{\ell}^2A_6q\right\rangle=-\left\langle q, A_6 q\right\rangle=0$. 
Since the first factor $\left\langle q, A_{\ell}A_6A_7q\right\rangle$ is not zero for almost every $q\in\mathbb{S}^7$, we see that the second factor 
vanishes almost everywhere on $\mathbb{S}^7$. 
As it is also a homogeneous quadratic polynomial in $q$, it must be identically zero for all $q\in\mathbb{R}^8$:
\begin{equation}\label{condition_reduced_2}
\sum_{i=1, \ldots, 5, i\neq\ell}\alpha_i\left\langle q, A_iA_{\ell}A_6A_7q\right\rangle-\sum_{j=1, \ldots, 5, j\neq\ell}\beta_{j7}\left\langle q, A_jA_{\ell}A_6q\right\rangle=0. 
\end{equation}
Now, we fix distinct numbers $m,n\in\{1,\ldots, 5\}\setminus \{\ell\}$.  By straightforward computations, we have 
\[
\left(A_mA_{\ell}A_6A_7\right)^2=\mathsf{E}_8, \quad \left(A_nA_{\ell}A_6\right)^2=\mathsf{E}_8, \quad \left[A_mA_{\ell}A_6A_7, A_nA_{\ell}A_6\right]=0.
\]
Since $A_mA_{\ell}A_6A_7$ and $A_nA_{\ell}A_6$ both anti-commute with $A_7$, they are different from $\pm \mathsf{E}_8$. 
Thus, the commuting symmetric matrices $A_mA_{\ell}A_6A_7$ and $A_nA_{\ell}A_6$ both have eigenvalues $\pm 1$ and 
$\mathbb{R}^8$ is decomposed into the direct sum of the simultaneous eigenspaces.
Hence, there exist simultaneous eigenvectors $q_{\pm}=q_{\pm}^{\ell mn}\in\mathbb{S}^7$ of $A_mA_{\ell}A_6A_7$ and $A_nA_{\ell}A_6$ such that $A_mA_{\ell}A_6A_7 q_{\pm}= q_{\pm}$ and $A_nA_{\ell}A_6q_{\pm}=\pm q_{\pm}$. 
In fact, we can take $q\in \mathbb{S}^7$ such that $A_mA_{\ell}A_6A_7 q= q$ and $A_nA_{\ell}A_6q=\epsilon q$, where either $\epsilon =1$ or $\epsilon =-1$. 
Then, 
\begin{align*}
A_mA_{\ell}A_6A_7\left(A_6A_7q\right)
&=
A_6A_7\left(A_mA_{\ell}A_6A_7q\right)
=
A_6A_7q, 
\\
A_nA_{\ell}A_6 \left(A_6A_7q\right)
&=
-A_6A_7\left(A_nA_{\ell}A_6q\right)
=
-\epsilon A_6A_7q, 
\end{align*}
and hence we can set $q_+=q$, $q_-=A_6A_7q$ or $q_+=A_6A_7 q$, $q_-=q$, respectively, depending on whether $\epsilon=1$ or $\epsilon=-1$. 
Now, for $i\in \left\{1, \ldots, 5\right\}$ such that $i\neq m$, $i\neq \ell$, we have 
\[
A_iA_{\ell}A_6A_7q_{\pm}
=-A_iA_m\left(A_mA_{\ell}A_6A_7 q_{\pm}\right)
=-A_iA_mq_{\pm} 
\]
and hence $\left\langle q_{\pm}, A_iA_{\ell}A_6A_7q_{\pm}\right\rangle =-\left\langle q_{\pm}, A_iA_mq_{\pm}\right\rangle=0$. 
Further, for $j\in \left\{1, \ldots, 5\right\}$ such that $j\neq n$, $j\neq \ell$, we have 
\[
A_jA_{\ell}A_6q_{\pm}
=
-A_jA_n\left(A_nA_{\ell}A_6q_{\pm}\right)
=
\mp A_jA_nq_{\pm} 
\]
and hence $\left\langle q_{\pm}, A_jA_{\ell}A_6q_{\pm}\right\rangle =\mp \left\langle q_{\pm}, A_jA_nq_{\pm}\right\rangle=0$. 
Thus, from \eqref{condition_reduced_2}, taking both eigenvectors $q_{\pm}$, we have 
\[
\alpha_m\pm \beta_{n7}=0, 
\]
which means that $\alpha_m=\beta_{n7}=0$. 
Since $\ell=1, \ldots, 5$ is arbitrary, this holds for all $m, n=1, \ldots, 5$. 
Similarly, assuming $\xi=A_{\ell}A_7q$, $\ell=1, \ldots, 5$, we have $\beta_{n6}=0$, $n=1, \ldots, 5$. 
Now, we have $A=\alpha_6 A_6+\alpha_7 A_7$. 
By $\left\{H_{\mathrm{sub},5}, F_A\right\}=0$, we have 
\begin{align*}
0&=
\left\langle A_6q, \xi\right\rangle\alpha_6\left\langle A_7A_6q, \xi\right\rangle+\left\langle A_7q, \xi\right\rangle\alpha_7\left\langle A_6A_7q, \xi\right\rangle \\
&=
\left\langle A_6A_7q, \xi\right\rangle\left(\alpha_6\left\langle A_6q, \xi\right\rangle-\alpha_7\left\langle A_7q, \xi\right\rangle\right), \quad \forall \,(q, \xi)\in T\mathbb{S}^7. 
\end{align*}
As $\left\langle A_6A_7q, \xi\right\rangle\neq 0$ almost everywhere on $T\mathbb{S}^7$, we have $\alpha_6\left\langle A_6q, \xi\right\rangle-\alpha_7\left\langle A_7q, \xi\right\rangle=0$ almost everywhere. 
As this is a homogeneous polynomial in $(q,\xi)$, we see that $\alpha_6\left\langle A_6q, \xi\right\rangle-\alpha_7\left\langle A_7q, \xi\right\rangle=0$ identically on $T\mathbb{S}^7$. 
Setting $\xi=A_6q$, we have $\alpha_6=0$. 
Then, we have $\left\langle A_7q, \xi\right\rangle\alpha_7\left\langle A_6A_7q, \xi\right\rangle=0$, which means $\alpha_7=0$. 
Therefore,  we have $A=0$, which proves the theorem. 
\end{proof}
\subsection{SR isometry group of $\mathbb{S}_{\textup{T},4}^7$}
The SR isometry group $\mathrm{Iso}_{\textup{sub}}\left(\mathbb{S}_{\textup{T},4}^7\right)$ has been studied in \cite[Theorem 6.4]{BL}. 
It was observed that the action of $\mathrm{Iso}_{\textup{sub}}\left(\mathbb{S}_{\textup{T},4}^7\right)$ on $\mathbb{S}^7$ is not transitive. 
Nevertheless, we here give a description of the infinitesimal isometries in $O(8)$ and we present a lower bound of 
$\dim\left(\mathrm{Iso}_{\textup{sub}}\left(\mathbb{S}_{\textup{T},4}^7\right)\right)$. 
{
\begin{thm}
The Lie algebra of $\mathrm{Iso}_{\textup{sub}}\left(\mathbb{S}_{\textup{T},4}^7\right)\cap O(8)$ coincides with $\mathfrak{h}_5$ in \eqref{Lie-subalgebras-rank-4-case}. 
In particular, we have $\dim \left(\mathrm{Iso}_{\textup{sub}}\left(\mathbb{S}_{\textup{T},4}^7\right)\right)\geq \dim\mathfrak{h}_5=9$. 
\end{thm}
\begin{proof}
We show the following: 
\begin{enumerate}
\item If $A\in \mathfrak{h}_5$, then $\left\{H_{\textup{sub},4}, F_A\right\}=0$. 
\item If $A\in \mathfrak{h}_5^{\perp}\subset \mathfrak{so}(8)$ and $\left\{H_{\textup{sub},4}, F_A\right\}=0$, then $A=0$. 
\end{enumerate}

As $H_{\textup{sub},4}=\displaystyle \dfrac{1}{2}\sum_{i=1}^4F_{A_i}^2=\dfrac{1}{2}\left\|\xi\right\|^2-\dfrac{1}{2}\sum_{i=5}^7F_{A_i}^2$ and since $\left\|\xi\right\|^2$ 
Poisson commutes with $F_A$, the condition $\left\{H_{\textup{sub},4}, F_A\right\}=0$ is equivalent to 
\begin{equation}\label{reduced_poisson_commutativity_rk_4}
F_{A_5}F_{\left[A_5, A\right]}
+
F_{A_6}F_{\left[A_6, A\right]}
+
F_{A_7}F_{\left[A_7, A\right]}
=
0, 
\end{equation}
according to Lemma  \ref{com_rel_ham_f_a}. 

\noindent (1)\; 
Let $A\in\mathfrak{h}_5$ be expanded as:
\[
A=\displaystyle \sum_{1\leq j<k\leq 4}\beta_{jk}A_jA_k+\sum_{5\leq j<k\leq 7}\beta_{jk}A_jA_k,
\]
where $\beta_{jk}\in \mathbb{R}$. 
Then, we have 
\begin{align*}
\left[A_5, A\right]
=
-2\beta_{56}A_6-2\beta_{57}A_7, \,
\left[A_6, A\right]
=
2\beta_{56}A_5-2\beta_{67}A_7, \,
\left[A_7, A\right]
=
2\beta_{57}A_5+2\beta_{67}A_6, 
\end{align*}
and hence 
\begin{align*}
F_{A_5}F_{\left[A_5, A\right]}
+
F_{A_6}F_{\left[A_6, A\right]}
+
F_{A_7}F_{\left[A_7, A\right]}
&=
-2\beta_{56}F_{A_5}F_{A_6}-2\beta_{57}F_{A_5}F_{A_7}+2\beta_{56}F_{A_6}F_{A_5} \\
&\quad -2\beta_{67}F_{A_6}F_{A_7}+2\beta_{57}F_{A_7}F_{A_5}+2\beta_{67}F_{A_7}F_{A_6}
=0. 
\end{align*}
This proves (1). 

\noindent (2)\; Let $A$ be in $\mathfrak{h}_5^{\perp}$. 
Then $A$ can be expanded in the form: 
\[
A=\sum_{i=1}^7\alpha_i A_i+\sum_{1\leq j\leq 4, 5\leq k\leq 7}\beta_{jk}A_jA_k,
\]
where $\alpha_i, \beta_{jk}\in\mathbb{R}$. 
Then, we have 
\begin{align*}
\left[A_{\ell}, A\right]
=
-2\sum_{i=1, i\neq \ell}^7\alpha_iA_iA_{\ell}+2\sum_{1\leq j\leq 4}\beta_{j\ell}A_j, 
\end{align*}
where $\ell=5,6,7$, and hence \eqref{reduced_poisson_commutativity_rk_4} can be written as 
\begin{align}\label{rk_4_A_5A_6A_7}
\sum_{\ell=5}^7\left\langle A_{\ell}q, \xi\right\rangle\cdot \left(\sum_{i=1,i\neq \ell}^7\alpha_i\left\langle A_iA_\ell q, \xi\right\rangle -\sum_{j=1}^4\beta_{j\ell}\left\langle A_jq, \xi\right\rangle\right)
=0. 
\end{align}
Now, we set $\xi=A_6A_7q$.
Using $\left\langle A_6q, \xi\right\rangle=\left\langle A_6q, A_6A_7q\right\rangle=\left\langle q, A_7q\right\rangle=0$ and $\left\langle A_7q, \xi\right\rangle=\left\langle A_7q, A_6A_7q\right\rangle=-\left\langle q, A_6q\right\rangle=0$, we have 
\[
\left\langle q, A_5A_6A_7 q\right\rangle\cdot \left(\sum_{i=1,i\neq 5}^7\alpha_i\left\langle q, A_iA_5A_6A_7 q\right\rangle -\sum_{j=1}^4\beta_{j5}\left\langle q, A_jA_6A_7q\right\rangle\right)=0.
\]
Since the first factor $\left\langle q, A_5A_6A_7 q\right\rangle$ is not zero for almost all $q\in \mathbb{S}^7$, the second factor vanishes for almost all $q\in \mathbb{S}^7$. 
As it is a homogeneous polynomial in $q$, we see that 
\begin{equation}\label{reduced_relation_for_isometry_rk_4}
\sum_{i=1,i\neq 5}^7\alpha_i\left\langle q, A_iA_5A_6A_7 q\right\rangle -\sum_{j=1}^4\beta_{j5}\left\langle q, A_jA_6A_7q\right\rangle=0
\end{equation}
for all $q\in \mathbb{R}^8$.  Fixing distinct numbers $\ell,m \in \{1,2,3,4\}$, we have the following: 
\[
\left(A_{\ell}A_5A_6A_7\right)^2=\mathsf{E}_8, \quad 
\left(A_mA_6A_7\right)^2=\mathsf{E}_8, \quad 
\left[A_{\ell}A_5A_6A_7, A_mA_6A_7\right]=0. 
\]
Therefore, we see that the symmetric matrices $A_{\ell}A_5A_6A_7$ and $A_mA_6A_7$ both have eigenvalues $\pm 1$ and that they induce simultaneous eigenspace decomposition of $\mathbb{R}^8$. 
Because of the commutation and the anti-commutation relations 
\begin{align*}
\left(A_{\ell}A_5A_6A_7\right)\cdot \left(A_5A_6\right)
&=
\left(A_5A_6\right)\cdot \left(A_{\ell}A_5A_6A_7\right), \\
\left(A_mA_6A_7\right)\cdot \left(A_5A_6\right)
&=
-\left(A_5A_6\right)\cdot \left(A_mA_6A_7\right),
\end{align*}
there exist $q_{\pm}=q_{\pm}^{\ell m}\in \mathbb{S}^7$ such that 
\[
A_{\ell}A_5A_6A_7q_{\pm}=q_{\pm}, \quad 
A_mA_6A_7q_{\pm}=\pm q_{\pm}. 
\]
For these $q_{\pm}$, \eqref{reduced_relation_for_isometry_rk_4} can be written as $\alpha_{\ell}\mp \beta_{m5}=0$ which means 
\[
\alpha_{\ell}=\beta_{m5}=0.
\]
As $\ell, m\in\{1,2,3,4\}$ are taken arbitrary with condition $\ell\neq m$, these relations hold for all $\ell, m\in\{1,2,3,4\}$. 
Due to the symmetry of \eqref{rk_4_A_5A_6A_7}, we have $\alpha_{\ell}=0$, $\ell=1, 2, 3, 4$, $\beta_{jk}=0$, $1\leq j\leq 4$, $5\leq k\leq 7$. 
We have $A=\alpha_5 A_5+\alpha_6 A_6+\alpha_7 A_7$. 
By $\left\{H_{\mathrm{sub},4}, F_A\right\}=0$, we have 
\begin{align}\label{condition_sub_4}
0&=
\left\langle A_6q, \xi\right\rangle\alpha_5\left\langle A_6A_5q, \xi\right\rangle+\left\langle A_7q, \xi\right\rangle\alpha_5\left\langle A_7A_5q, \xi\right\rangle \notag \\
&\quad +
\left\langle A_5q, \xi\right\rangle\alpha_6\left\langle A_5A_6q, \xi\right\rangle+\left\langle A_7q, \xi\right\rangle\alpha_6\left\langle A_7A_6q, \xi\right\rangle \notag \\
&\quad +
\left\langle A_5q, \xi\right\rangle\alpha_7\left\langle A_5A_7q, \xi\right\rangle+\left\langle A_6q, \xi\right\rangle\alpha_7\left\langle A_6A_7q, \xi\right\rangle, \quad \forall (q, \xi)\in T\mathbb{S}^7. 
\end{align}
Setting $\xi=A_4A_5q$, we have 
\[
\left\langle A_6 A_7q, A_4A_5q\right\rangle\left(\alpha_6\left\langle A_7q, A_4A_5q\right\rangle-\alpha_7\left\langle A_6q, A_4A_5q\right\rangle\right)=0. 
\]
Since $\left\langle A_6 A_7q, A_4A_5q\right\rangle\neq 0$ almost everywhere on $T\mathbb{S}^7$, we have $\alpha_6\left\langle A_7q, A_4A_5q\right\rangle-\alpha_7\left\langle A_6q, A_4A_5q\right\rangle=0$ almost everywhere. 
As it is a homogeneous polynomial in $(q,\xi)$, we see that $\alpha_6\left\langle A_7q, A_4A_5q\right\rangle-\alpha_7\left\langle A_6q, A_4A_5q\right\rangle=0$ everywhere on $T\mathbb{S}^7$. 
Now we assume that $q$ is an eigenvector of $A_4A_5A_7$. 
(Recall that $A_4A_5A_7$ has eigenvalues $\pm 1$.) 
Then, we have $\alpha_6=0$ and hence $\alpha_7=0$. 
Finally, we have $0=
\left\langle A_6q, \xi\right\rangle\alpha_5\left\langle A_6A_5q, \xi\right\rangle+\left\langle A_7q, \xi\right\rangle\alpha_5\left\langle A_7A_5q, \xi\right\rangle$ by \eqref{condition_sub_4}. 
By a similar argument, we have $\alpha_5=0$. 
Thus, we have $A=0$. 
\end{proof}
\subsection{SR isometry group of $\mathbb{S}_{\textup{T},6}^7=\mathbb{S}_{\textup{H}}^7$}
We consider the SR isometry group $\mathrm{Iso}_{\textup{sub}}\left(\mathbb{S}_{\textup{T},6}^7\right)=\mathrm{Iso}_{\textup{sub}}\left(\mathbb{S}_{\textup{H}}^7\right)$. 
\begin{thm}\label{inf_iso_cork_one}
The Lie algebra of $\mathrm{Iso}_{\textup{sub}}\left(\mathbb{S}_{\textup{H}}^7\right)\cap O(8)$ coincides with $\mathfrak{k}_6$ in \eqref{Lie-subalgebras-rank-6-case}. 
In particular, we have $\dim\left(\mathrm{Iso}_{\textup{sub}}\left(\mathbb{S}_{\textup{H}}^7\right)\right)\geq \dim \mathfrak{k}_6=16$. 
\end{thm}
\begin{proof}
We show the followings: 
\begin{enumerate}
\item If $A\in \mathfrak{k}_6$, then $\left\{H_{\textup{sub},6}, F_A\right\}=0$. 
\item If $A\in \mathfrak{k}_6^{\perp}$ and $\left\{H_{\textup{sub},6}, F_A\right\}=0$, then $A=0$. 
\end{enumerate}

\noindent (1)\; 
Let $A$ be in $\mathfrak{k}_6$. 
By Lemma \ref{com_rel_ham_f_a} we have 
\[
\left\{H_{\textup{sub},6}, F_A\right\}=0
\iff
F_{[A_7,A]}=0
\iff 
[A_7,A]=0. 
\]
The last equality follows from $[A_7, A_jA_k]=0$ for $1\leq j<k\leq 6$. 

\noindent (2)\; 
Assume that $A=\displaystyle \sum_{i=1}^6\alpha_iA_i+\sum_{j=1}^6\beta_{j7}A_jA_7\in \mathfrak{k}_6^{\perp}$, $\alpha_i, \beta_{j7}\in\mathbb{R}$, satisfies $[A_7,A]=0$. 
Then we have 
\[
0=
\left[A_7,A\right]
=
-2\sum_{i=1}^6\alpha_iA_iA_7+2\sum_{j=1}^6\beta_{j7}A_j.
\]
Since $A_i$, $i=1, \ldots, 6$, $A_jA_7$, $j=1, \ldots, 6$, are linearly independent, it follows $\alpha_i=0$, $i=1, \ldots, 6$, $\beta_{j7}=0$, $j=1, \ldots, 6$, i.e. $A=0$. 
\end{proof}
}
Note that $\mathfrak{k}_6$ contains five elements, e.g. $A_1A_j$, $j=2, \ldots, 6$, satisfying the Clifford relation \eqref{GL-Matrix-Clifford-relations}. 
Using Theorem \ref{inf_iso_cork_one}, an argument similar to \cite[Proposition 3.5]{BLT} and \cite{BFI}, as well as to Corollary \ref{cor_g_5}, shows that the SR isometry group $\mathrm{Iso}_{\textup{sub}}\left(\mathbb{S}_{\textup{T},6}^7\right)=\mathrm{Iso}_{\textup{sub}}\left(\mathbb{S}_{\textup{H}}^7\right)$ acts transitively on $\mathbb{S}^7$. 

\subsection{SR isometry group of $\mathbb{S}_{\textup{QH}}^7$}
We consider the SR isometry group $\mathrm{Iso}_{\textup{sub}}\left(\mathbb{S}_{\textup{QH}}^7\right)$.
We start by determining the infinitesimal isometries in $O(8)$. 
\begin{thm}\label{thm_iso_sub_qh}
The Lie algebra of $\mathrm{Iso}_{\textup{sub}}\left(\mathbb{S}_{\textup{QH}}^7\right)\cap O(8)$ coincides with $\mathfrak{g}_6$ in \eqref{chain_subalgebras_1}. 
In particular, we have $\dim \left(\mathrm{Iso}_{\textup{sub}}\left(\mathbb{S}_{\textup{QH}}^7\right)\right)\geq \dim\mathfrak{g}_6=13$. 
\end{thm}
\begin{proof}
We show the following: 
\begin{enumerate}
\item If $A\in \mathfrak{g}_6$, then $\left\{H_{\textup{sub,QH}}, F_A\right\}=0$. 
\item If $A\in \mathfrak{g}_6^{\perp}\subset\mathfrak{so}(8)$ and $\left\{H_{\textup{sub,QH}}, F_A\right\}=0$, then $A=0$. 
\end{enumerate}

Since $H_{\textup{sub,QH}}=\displaystyle\dfrac{1}{2}\left\|\xi\right\|^2-\dfrac{1}{2}\left(F_{A_6}^2+F_{A_7}^2+F_{A_6A_7}^2\right)$ and $\left\|\xi\right\|^2$ 
Poisson commutes  with $F_A$,  the condition $\left\{H_{\textup{sub,QH}}, F_A\right\}=0$ is, by Lemma \ref{com_rel_ham_f_a}, equivalent to 
\begin{equation}\label{reduced_poisson_commutativity_Q}
F_{A_6}F_{\left[A_6, A\right]}
+
F_{A_7}F_{\left[A_7, A\right]}
+
F_{A_6A_7}F_{\left[A_6A_7, A\right]}
=
0. 
\end{equation}

\noindent (1)\; Let $A\in\mathfrak{g}_6$ be expanded as:   
\[
A=\displaystyle \alpha_6A_6+\alpha_7A_7+\sum_{1\leq j<k\leq 5}\beta_{jk}A_jA_k+\beta_{67}A_6A_7,
\]
where $\alpha_i, \beta_{jk}\in \mathbb{R}$. 
We have 
\begin{align*}
\left[A_6, A\right]
&=
2\alpha_7A_6A_7-2\beta_{67}A_7, 
\,
\left[A_7, A\right]
=
-2\alpha_6A_6A_7+2\beta_{67}A_6, 
\\
\left[A_6A_7, A\right]
&=
2\alpha_6A_7-2\alpha_7A_6, 
\end{align*}
and hence 
\begin{align*}
&F_{A_6}F_{\left[A_6, A\right]}
+
F_{A_7}F_{\left[A_7, A\right]}
+
F_{A_6A_7}F_{\left[A_6A_7, A\right]} \\
&=
2\alpha_7F_{A_6}F_{A_6A_7}-2\beta_{67}F_{A_6}F_{A_7}-2\alpha_6F_{A_7}F_{A_6A_7} \\
&\quad 
+2\beta_{67}F_{A_7}F_{A_6}+2\alpha_6F_{A_6A_7}F_{A_7}-2\alpha_7F_{A_6A_7}F_{A_6}
=0. 
\end{align*}
This proves (1). 

\noindent (2)\; 
We expand $A\in \mathfrak{g}_6$ as 
\[
A
=
\displaystyle \sum_{i=1}^5\alpha_i A_i+\sum_{1\leq j\leq 5, 6\leq k\leq 7}\beta_{jk}A_jA_k, 
\]
where $\alpha_i, \beta_{jk}\in\mathbb{R}$. Note that  
\begin{align*}
\left[A_\ell, A\right]
&=
-2\sum_{i=1}^5\alpha_iA_iA_\ell+2\sum_{j=1}^5\beta_{j\ell}A_j, \quad \ell=6,7; 
\\
\left[A_6A_7, A\right]
&=
2\sum_{j=1}^5\beta_{j6}A_jA_7-2\sum_{j=1}^5\beta_{j7}A_jA_6, 
\end{align*}
and hence \eqref{reduced_poisson_commutativity_Q} can be computed as 
\begin{align}\label{Q_A_6A_7A_6A_7}
&\sum_{\ell=6}^7
\left\langle A_\ell q, \xi\right\rangle\cdot \left(\sum_{i=1}^5\alpha_i\left\langle A_iA_\ell q, \xi\right\rangle -\sum_{j=1}^5\beta_{j\ell}\left\langle A_jq, \xi\right\rangle\right) \notag \\
&+\left\langle A_6A_7q, \xi\right\rangle\cdot \left(-\sum_{j=1}^5\beta_{j6}\left\langle A_jA_7q, \xi\right\rangle +\sum_{j=1}^5\beta_{j7}\left\langle A_jA_6q, \xi\right\rangle\right)
=0. 
\end{align}

Now, we fix $\ell\in\{1, \ldots, 5\}$ and set $\xi=A_{\ell}q$. 
Then, we see that the double sum in \eqref{Q_A_6A_7A_6A_7} vanishes. 
Therefore, we have 
\[
\left\langle A_6A_7q, A_{\ell}q\right\rangle\cdot \left(\sum_{j=1}^5\beta_{j6}\left\langle A_jA_7q, A_{\ell}q\right\rangle -\sum_{j=1}^5\beta_{j7}\left\langle A_jA_6q, A_{\ell}q\right\rangle\right)
=0. 
\]
The first factor $\left\langle A_6A_7q, A_{\ell}q\right\rangle=-\left\langle q, A_6A_7A_{\ell}q\right\rangle$ does not vanish for almost all $q\in \mathbb{S}^7$. 
Since the second factor is a homogeneous polynomial in $q$, we have 
\begin{equation}\label{reduced_relation_for_isometry_Q}
\sum_{j=1, j\neq \ell}^5\beta_{j6}\left\langle q, A_{\ell}A_jA_7q\right\rangle -\sum_{j=1, j\neq \ell}^5\beta_{j7}\left\langle q, A_{\ell}A_jA_6q\right\rangle
=0 
\end{equation}
for all $q\in \mathbb{R}^8$. 

Next, for fixed distinct numbers $m,n\in \{1,\ldots, 5\}\setminus \{\ell\}$, we have 
\[
\left(A_{\ell}A_mA_7\right)^2=\mathsf{E}_8, \quad 
\left(A_{\ell}A_nA_6\right)^2=\mathsf{E}_8, \quad
\left[A_{\ell}A_mA_7, A_{\ell}A_nA_6\right]=0, 
\]
and hence the symmetric matrices $A_{\ell}A_mA_7$ and $A_{\ell}A_nA_6$ both have eigenvalues $\pm 1$.  Therefore they admit a simultaneous eigenspace decomposition of $\mathbb{R}^8$. 
Since 
\[
\left(A_{\ell}A_mA_7\right)\cdot A_7
=
A_7 \cdot \left(A_{\ell}A_mA_7\right), \quad 
\left(A_{\ell}A_nA_6\right)\cdot A_7
=
-A_7\cdot \left(A_{\ell}A_nA_6\right),
\]
there exist two vectors $q_{\pm}=q_{\pm}^{\ell mn}\in \mathbb{S}^7$ such that 
\[
A_{\ell}A_mA_7q_{\pm}=q_{\pm}, \quad 
A_{\ell}A_nA_6q_{\pm}=\pm q_{\pm}. 
\]
For these $q_{\pm}$, \eqref{reduced_relation_for_isometry_Q} can be written as $\beta_{m6}\mp \beta_{n7}=0$, which means that
\[
\beta_{m6}=\beta_{n7}=0.
\]
As distinct numbers $\ell, m, n\in\{1, \ldots, 5\}$ are taken arbitrary, these relations hold for all $\ell, m\in\{1, \ldots, 5\}$. 

Returning to \eqref{Q_A_6A_7A_6A_7}, we now have 
\[
\left\langle A_6q, \xi\right\rangle\cdot \sum_{i=1}^5\alpha_i\left\langle A_iA_6q, \xi\right\rangle
+
\left\langle A_7q, \xi\right\rangle\cdot \sum_{i=1}^5\alpha_i\left\langle A_iA_7q, \xi\right\rangle
=0. 
\]
Setting $\xi=A_{\ell}A_6q$ for some (fixed) $\ell\in\{1, \ldots, 5\}$, we see that the first term vanishes because $\left\langle A_6q, A_{\ell}A_6q \right\rangle=-\left\langle q, A_{\ell}q\right\rangle=0$. 
Hence, 
\[
\left\langle A_7q, A_{\ell}A_6q\right\rangle\cdot \sum_{i=1, i\neq \ell}^5\alpha_i\left\langle A_iA_7q, A_{\ell}A_6q\right\rangle
=0. 
\]
Note that the first factor $\left\langle A_7q, A_{\ell}A_6q\right\rangle=-\left\langle q, A_{\ell}A_6A_7q\right\rangle$ is not zero for almost all $q\in \mathbb{S}^7$. 
Since the second factor is a homogeneous polynomial in $q$, we have 
\[
\sum_{i=1,i\neq \ell}^5\alpha_i\left\langle q, A_iA_{\ell}A_6A_7q\right\rangle
=0
\]
for all $q\in \mathbb{R}^8$. 
We fix $m\in\{1, \ldots, 5\}\setminus \{\ell\}$.  As $\left(A_mA_{\ell}A_6A_7\right)^2=\mathsf{E}_8$ and $A_mA_{\ell}A_6A_7\neq \pm \mathsf{E}_8$, there exists a vector 
$q\in \mathbb{S}^7$ such that $A_mA_{\ell}A_6A_7q=q$. 
In this case, if $i\in\{1, \ldots, 5\}\setminus \{\ell,m\}$, we have $A_iA_{\ell}A_6A_7q=-A_iA_mA_mA_{\ell}A_6A_7q=-A_iA_mq$ and hence $\left\langle q, A_iA_{\ell}A_6A_7q\right\rangle =-\left\langle q, A_iA_mq\right\rangle=0$. 
Therefore, we have $\alpha_m=0$. 
Since the distinct numbers $\ell, m \in \{1, \ldots, 5\}$ are arbitrary, this is satisfied for all $m=1, \ldots, 5$. 
Thus, we have $A=0$. 
\end{proof}

\begin{rem}
By Theorem \ref{thm_iso_sub_qh}, we see that the Lie group $G_5\subset O(8)$ corresponding to the Lie algebra $\mathfrak{g}_5$ is included in the SR isometry group $\mathrm{Iso}_{\textup{sub}}\left(\mathbb{S}_{\textup{QH}}^7\right)$. 
By Corollary \ref{cor_g_5}, we conclude that $\mathrm{Iso}_{\textup{sub}}\left(\mathbb{S}_{\textup{QH}}^7\right)$ acts transitively on $\mathbb{S}^7$. 
\end{rem}
\section{Commuting differential operators}
One method of determining a family of Poisson commuting first integrals consists in constructing a set of pseudo-differential operators including the sublaplacian, which pairwise commute modulo first order operators, and considering their principal symbols. 
In the present section, we show that, for all of the above SR structures on $\mathbb{S}^7$, an even stronger and converse statement is true, as well. 
In each case, we define a set of seven commuting differential operators including the sublaplacian, with full symbols being the previously constructed Poisson commuting first integrals. 

\medskip

\noindent
{\bf Rank 5 trivializable structure:} We assign to each of the subalgebras $\mathfrak{g}_{\ell}$ in (\ref{chain_subalgebras_1}) a second order differential operator $\mathcal{G}_{\ell}$ as follows: 
\begin{equation}\label{differential-operator-assigned-to-subalgebra}
\mathcal{G}_{\ell}:= \sum_{E \in \mathfrak{G}_{\ell}} X(E)^2 \hspace{3ex} \mbox{\it where} \hspace{3ex} \ell=1, \ldots, 7. 
\end{equation}
Here we use the notation analogous to \eqref{canonical_vf} and by $\mathfrak{G}_{\ell}$ we denote the set of generators of $\mathfrak{g}_{\ell}$ in \eqref{chain_subalgebras_1}. 
\begin{prop}\label{Proposition_commuting_differential_operators}
The operators $\mathcal{G}_{\ell}$ for $\ell=1, \ldots, 7$ pairwise commute and commute with the sublaplacian $\Delta_{\textup{T},5}^{\textup{sub}}=-\sum_{j=1}^5X(A_j)^2$ in \eqref{Definition_sub_laplacians}. 
\end{prop} 
\begin{proof}
We first treat the case $\ell=1, \ldots, 4$. Note that 
\begin{equation*}
\mathcal{G}_{\ell} = \mathcal{G}_{n}+ \sum_{j=n+2}^{\ell+1}\sum_{i=1}^{j-1} X(A_iA_j)^2
\hspace{3ex}\mbox{\it where} \hspace{3ex} 1\leq n < \ell=2, \ldots ,4. 
\end{equation*}
We obtain the commutator: 
\begin{align*}
\big{[} \mathcal{G}_{\ell}, \mathcal{G}_n \big{]} 
&= \sum_{j=n+2}^{\ell+1} \sum_{i=1}^{j-1} \big{[}X(A_iA_j)^2, \mathcal{G}_n \big{]}.
\end{align*}
Note that $[X(A_iA_j), X(A_rA_k)]=-2X([A_iA_j,A_rA_k]))=0$ for pairwise distinct indices $\{i,j,r,k\}$ (see e.g. \cite{BFI}). 
Moreover, let 
$1 \leq i< r \leq  n+1$ and $n+2 \leq j \leq \ell+1$, then 
\begin{align}
&\big{[}X(A_iA_j)^2+X(A_rA_j)^2, X(A_iA_r) \big{]}\label{eq_commutators_sum_of_squares}\\
=&X(A_iA_j) \big{[} X(A_iA_j), X(A_iA_r) \big{]} + \big{[} X(A_iA_j), X(A_iA_r)\big{]} X(A_iA_j)\notag \\
& \hspace{4ex} +X(A_rA_j) \big{[}X(A_rA_j), X(A_iA_r) \big{]} + \big{[} X(A_rA_j), X(A_iA_r) \big{]} X(A_rA_j) \notag \\
& \hspace{4ex}= -2X(A_iA_j) X(A_jA_r) -2X(A_jA_r)X(A_iA_j)\notag \\
& \hspace{8ex} +2X(A_rA_j)X(A_jA_i)+2X(A_jA_i)X(A_rA_j)=0. \notag 
\end{align}
In conclusion, $[\mathcal{G}_{\ell}, \mathcal{G}_n]=0$ for $1\leq n < \ell$ and $\ell=2, \ldots, 4$. The operator $\mathcal{G}_5$ is decomposed  as 
$$\mathcal{G}_5= \mathcal{G}_4+ X(A_6A_7)^2,$$ 
where the last summand commutes with $\mathcal{G}_{\ell}$ for $\ell=1, \ldots, 4$ as well as with the sum-of-squares $X(A_6)^2+X(A_7)^2$. 
Hence $\mathcal{G}_5$ commutes with $\mathcal{G}_1, \ldots, \mathcal{G}_6$. Note that $X(A_6)$ and $X(A_7)$ commute with $X(A_iA_j)$ where 
$1\leq i<j \leq 5$ and therefore $\mathcal{G}_6$ commutes with $\mathcal{G}_1, \ldots, \mathcal{G}_4$. 
\vspace{1mm}\par 
Finally, we have 
\begin{equation}\label{decomposition-of-L-7}
\mathcal{G}_7= \mathcal{G}_5+ \sum_{j=6}^7 \sum_{i=1}^{j-1} X(A_iA_j)^2 - \Delta_{\mathbb{S}^7}. 
\end{equation}
Note that $A_i$ and $A_iA_j$ act as isometries on $\mathbb{S}^7$ for all $i,j=1, \ldots, 7$ and hence the standard Laplacian $\Delta_{\mathbb{S}^7}=-\sum_{j=1}^7X(A_j)^2$ commutes with the Killing vector fields $X(A_i)$ and $X(A_iA_j)$. Therefore, $\Delta_{\mathbb{S}^7}$ commutes with $\mathcal{G}_{\ell}$ where $\ell=1, \ldots, 7$. 
It remains to show that 
\begin{equation*}
 \sum_{j=6}^7 \sum_{i=1}^{j-1} \big{[}X(A_iA_j), \mathcal{G}_{\ell}\big{]} =0, \hspace{4ex} \ell =1, \ldots, 5.
\end{equation*}
This follows by a calculation similar to (\ref{eq_commutators_sum_of_squares}). 
\end{proof}
\noindent
{\bf Quaternionic Hopf fibration:} Next we consider the sublaplacian $\Delta_{\textup{QH}}^{\textup{sub}}$ on $\mathbb{S}_{\textup{QH}}^7$ induced by the 
quaternionic Hopf fibration. 
More precisely, 
\begin{align} \label{decomposition_Delta_quaterionic-Hopf}
\Delta_{\textup{QH}}^{\textup{sub}}
&=\Delta_{\mathbb{S}^7} +  X(A_6)^2+X(A_7)^2+X(A_6A_7)^2 
= \Delta_{\textup{sub}, 5} + X(A_6A_7)^2. 
\end{align}
\begin{prop}\label{Prop-commuting-differential-operators-4}
The operators $\mathcal{G}_{\ell}$ for $\ell=1,\ldots, 7$ in Proposition \ref{Proposition_commuting_differential_operators} pairwise commute and commute with 
$\Delta_{\textup{QH}}^{\textup{sub}}$. 
\end{prop}
\begin{proof}
The first statement follows from Proposition \ref{Proposition_commuting_differential_operators}. It remains to show that the operators 
$\mathcal{G}_{\ell}$ for $\ell=1, \ldots, 7$ commute with $X(A_6A_7)^2$. In fact, this follows from $X(A_6A_7)^2= \mathcal{G}_5- \mathcal{G}_4$ and the first statement. 
\end{proof}

\noindent
{\bf Rank 4 trivializable structure:} We assign to each of the subalgebras in (\ref{Lie-subalgebras-rank-4-case}) a differential operator $\mathcal{H}_{\ell}$ as follows:
\begin{equation*}
\mathcal{H}_{\ell}:= \sum_{E \in \mathfrak{H}_{\ell}} X(E)^2 \hspace{3ex} \mbox{\it where} \hspace{3ex} \ell=1, \ldots, 7. 
\end{equation*}
Here $\mathfrak{H}_{\ell}$ denotes the set of generators of the Lie algebra $\mathfrak{h}_{\ell}$ in \eqref{Lie-subalgebras-rank-4-case}. 
\begin{prop}\label{Prop-commuting-differential-operators-3}
The operators $\mathcal{H}_{\ell}$ for $\ell=1,\ldots, 7$ pairwise commute. Moreover, they commute with the sublaplacian 
$\Delta_{\textup{T},4}^{\textup{sub}} =- \sum_{j=1}^4X(A_j)^2$. 
\end{prop}
\begin{proof} 
Note that $\mathfrak{g}_{\ell}= \mathfrak{h}_{\ell}$ and therefore $\mathcal{G}_{\ell}= \mathcal{H}_{\ell}$ for $\ell=1,2,3$. Proposition 
\ref{Proposition_commuting_differential_operators} implies that $\mathcal{H}_{\ell}$ for $\ell=1,2,3$ pairwise commute. Moreover, 
\begin{equation*}
\mathcal{H}_4= \mathcal{H}_3+ X(A_5A_6)^2.
\end{equation*}
Since $[X(A_5A_6), X(A_iA_j)]=0$, where $1\leq i<j\leq 4$ one concludes that $\mathcal{H}_4$ commutes with $\mathcal{H}_i$ for $i=1,2,3$. We have:
\begin{equation*}
\mathcal{H}_5= \mathcal{H}_3+ X(A_5A_6)^2+X(A_5A_7)^2+X(A_6A_7)^2
\end{equation*}
showing that $\mathcal{H}_5$ commutes with $\mathcal{H}_i$ for $i=1, \cdots, 3$. A calculation similar to 
(\ref{eq_commutators_sum_of_squares}) shows that 
\begin{equation}\label{commutator-sum-of-squares-rank-4}
\big{[} X(A_5A_7)^2+X(A_6A_7)^2, X(A_5A_6) \big{]} =0. 
\end{equation}
Hence $\mathcal{H}_5$ commutes with $\mathcal{H}_4$. Observe that 
$\mathcal{H}_6=\mathcal{H}_5 + X(A_5)^2+X(A_6)^2+X(A_7)^2$
commutes with $\mathcal{H}_i$, where $i=1, \ldots, 5$. 
The operator $\mathcal{H}_7$ decomposes as: 
\begin{equation}\label{decomposition_J_7}
\mathcal{H}_7=\mathcal{H}_6+\sum_{i=1}^4\sum_{j=5}^7 X(A_iA_j)^2+\sum_{k=1}^4X(A_k)^2. 
\end{equation}
Let $1 \leq \ell< r \leq 4$, or $5 \leq m < n \leq 7$, then 
\begin{align*}
\Big{[} X(A_{\ell}A_r), \sum_{i=1}^4 \sum_{j=5}^7 X(A_iA_j)^2 \Big{]}
&= \Big{[} X(A_{\ell}A_r), \sum_{j=5}^7 \big{(} X(A_{\ell}A_j)^2 +X(A_rA_j)^2 \big{)} \Big{]}=0, \\
\Big{[} X(A_mA_n), \sum_{i=1}^4 \sum_{j=5}^7 X(A_iA_j)^2 \Big{]}
&= \Big{[} X(A_mA_n), \sum_{i=1}^4 \big{(} X(A_iA_m)^2+X(A_iA_n)^2 \big{)} \Big{]}=0,
\end{align*}
by a calculation similar to (\ref{eq_commutators_sum_of_squares}). Finally, note that 
\begin{align*}
\Big{[} \sum_{k=5}^7X(A_k)^2, \sum_{k=1}^4X(A_k)^2 \Big{]}&= -\sum_{k=1}^4\Big{[} \Delta_{\mathbb{S}^7}, X(A_k)^2 \Big{]}
=0, \\
\Big{[} X(A_mA_n), \sum_{k=1}^4X(A_k)^2\Big{]}
&=0, \\
\Big{[}X(A_{\ell}A_r),  \sum_{k=1}^4X(A_k)^2\Big{]}
&= -\Big{[}X(A_{\ell}A_r), \Delta_{\mathbb{S}^7} \Big{]} +\Big{[}X(A_{\ell}A_r), \sum_{k=5}^7 X(A_k)^2 \Big{]}=0. 
\end{align*} 
Hence, $\mathcal{H}_7$ commutes with $\mathcal{H}_i$, where $i=1, \ldots, 6$. By similar calculations it follows that all operators $\mathcal{H}_i$ for $i=1,\ldots, 7$ commute with the sublaplacian $\Delta_{\textup{T},4}^{\textup{sub}}$. 
\end{proof}
\noindent
{\bf Rank 6 trivializable structure (Hopf fibration):} We assign to each of the subalgebras $\mathfrak{k}_{\ell}$ in (\ref{Lie-subalgebras-rank-6-case}) a 
second order differential operators $\mathcal{K}_{\ell}$ as follows: 
\begin{equation*}
\mathcal{K}_{\ell}:= \sum_{E \in\mathfrak{K}_{\ell}}X(E)^2 \hspace{3ex} \mbox{\it where} \hspace{3ex} \ell=1, \ldots, 7.
\end{equation*}
Here $\mathfrak{K}_{\ell}$ denotes the set of generators of the Lie algebra $\mathfrak{k}_{\ell}$ in \eqref{Lie-subalgebras-rank-6-case}. 
\begin{prop}\label{Prop-commuting-differential-operators-5}
The operators $\mathcal{K}_{\ell}$ for $\ell=1,\ldots, 7$ pairwise commute. Moreover, they commute with the sublaplacian 
$\Delta_{\textup{T},6}^{\textup{sub}} = -\sum_{j=1}^6X(A_j)^2$. 
\end{prop}
\begin{proof} 
Note that $\mathfrak{k}_{\ell}= \mathfrak{g}_{\ell}$ and therefore $\mathcal{K}_{\ell}= \mathcal{G}_{\ell}$ for $\ell=1,\ldots, 4$. Hence the proof of Proposition \ref{Proposition_commuting_differential_operators} implies that $\mathcal{K}_{\ell}$ for $\ell=1,\ldots, 4$ pairwise commute. Moreover, the operator $\mathcal{K}_5$ 
decomposes as: 
\begin{equation*}
\mathcal{K}_5=\mathcal{K}_4+\sum_{k=1}^5X(A_kA_6)^2. 
\end{equation*}
Let $1 \leq i <j\leq 5$, then 
\begin{equation*}
\Big{[}X(A_iA_j), \sum_{k=1}^5X(A_kA_6)^2\Big{]}= \Big{[} X(A_iA_j), X(A_iA_6)^2+X(A_jA_6)^2\Big{]}=0
\end{equation*}
by a calculation similar to (\ref{eq_commutators_sum_of_squares}). Hence $\mathcal{K}_5$ commutes with $\mathcal{K}_1, \ldots, \mathcal{K}_4$. Note that 
$\mathcal{K}_6= \mathcal{K}_5+X(A_7)^2$ obviously commutes with $\mathcal{K}_1,\ldots, \mathcal{K}_5$. 
Finally, 
\begin{equation*}
\mathcal{K}_7= \mathcal{K}_6+ \sum_{\ell=1}^6 X(A_{\ell}A_7)^2 + \sum_{k=1}^6 X(A_k)^2= \mathcal{K}_5- \Delta_{\mathbb{S}^7}+ \sum_{\ell=1}^6 X(A_{\ell}A_7)^2 . 
\end{equation*}
Hence, $\mathcal{K}_1, \ldots, \mathcal{K}_7$ pairwise commute and also commute with the sublaplacian $\Delta_{\textup{T},6}^{\textup{sub}}= \Delta_{\mathbb{S}^7}+X(A_7)^2$. 

\end{proof} 
Propositions \ref{Proposition_commuting_differential_operators}, \ref{Prop-commuting-differential-operators-4}, \ref{Prop-commuting-differential-operators-3} and  
\ref{Prop-commuting-differential-operators-5} induce in each of the cases a number of Poisson commuting first integrals adapted to the underlying subriemannian structure. 
\begin{thm}
The (principal) symbols of the operators $\mathcal{G}_{\ell}$ with $\ell=1, \ldots, 7$ are in involution and Poisson commute  with the SR Hamiltonian $H_{\textup{sub}, 5}$. The analogous result holds for the operators  $\mathcal{H}_{\ell}$, $\mathcal{K}_{\ell}$ 
where $\ell=1, \ldots, 7$ and the (principal) symbols of the corresponding sublaplacians. 
\end{thm}
\begin{proof}
The (principal) symbol of the commutator between two of the operators $\mathcal{G}_{\ell}$ vanishes. It coincides with the Poisson bracket of the operator symbols. 
\end{proof}

\medskip

\noindent {\bf Acknowledgements.}\; 
This work is partially supported by Institut f\"{u}r Analysis, Leibniz Universit\"{a}t Hannover; Research Institute for Mathematical Sciences, an International Joint Usage/Research Center located in Kyoto University; and Ritsumeikan University. 

{

\end{document}